\documentclass[11pt]{amsart}
\usepackage{verbatim,amssymb,amsmath,graphicx,hyperref,color,float,comment}

  \scrollmode

\newcommand{\R}{{\mathbb R}}
\newcommand{\N}{{\mathbb N}}

\newcommand{\E}{{\mathbb E}}
\newcommand{\EE}{{\mathbb E}}
\newcommand{\PP}{{\mathbb P}}
\newcommand{\prob}{{\mathbb P}}

\newcommand{\eul}{{\widehat X}}

\newcommand{\usn}{\underline {s}_n}
\newcommand{\utn}{\underline {t}_n}

\newcommand{\sgn}{\operatorname{sgn}}
\newcommand{\eps}{\varepsilon}
\renewcommand{\epsilon}{\varepsilon}
\newcommand{\F}{{\mathcal F}}

\newcommand{\reach}{\text{reach}}
\newcommand{\normal}{\mathfrak{n}}
\newcommand{\nor}{\mathfrak{n}}
\newcommand{\pr}{\text{pr}}
\newcommand{\unp}{\text{unp}}
\newcommand{\tr}{\text{tr}}
\newcommand{\cl}{\text{cl}}
\newcommand{\inter}{\text{int}}

\theoremstyle{plain}
\newtheorem{Thm}{Theorem}
\newtheorem{Prop}{Proposition}
\newtheorem{Lem}{Lemma}
\newtheorem{Cor}{Corollary}

\theoremstyle{definition}
\newtheorem{Rem}{Remark}
\newtheorem{exs}{Example}

\newcommand{\app}[1]{\textcolor{black}{#1}}

\usepackage{mathtools}
\DeclarePairedDelimiter{\norm}{\|}{\|}
\DeclarePairedDelimiter{\abs}{\lvert}{\rvert}

\DeclarePairedDelimiter{\floor}{\lfloor}{\rfloor}
\usepackage{dsfont}
\newcommand{\one}{\mathds{1}}

\begin{document}
\title[
Milstein-type methods for  SDEs with discontinuous drift coefficient]{
Milstein-type methods for strong approximation of
systems of SDEs with a discontinuous drift coefficient}

\author[Rauh\"ogger]
{Christopher Rauh\"ogger}
\address{
	Faculty of Computer Science and Mathematics\\
	University of Passau\\
	Innstrasse 33 \\
	94032 Passau\\
	Germany} \email{christopher.rauhoegger@uni-passau.de}

\begin{abstract} 
	We study strong approximation of $d$-dimensional stochastic differential equations (SDEs) with a discontinuous drift coefficient driven by a $d$-dimensional 
	Brownian motion $W$.
	More precisely, we essentially assume that the drift coefficient $\mu$ is piecewise Lipschitz continuous with an exceptional set $\Theta\subset \R^d$
	that is an orientable $C^5$-hypersurface of positive reach, the diffusion coefficient $\sigma$ is  assumed to be 
	Lipschitz continuous and, in a neighborhood of $\Theta$, both coefficients are bounded and $\sigma$ is non-degenerate.
	Furthermore, both $\mu$ and $\sigma$ are assumed to be $C^{1}$ with intrinsic Lipschitz continuous derivative on $\R^{d}\setminus \Theta$.

	In recent years, a number of results have been proven in the literature for
	strong approximation of such SDEs. In \cite{MGYR24} it was shown that the 
	Euler-Maruyama scheme achieves an $L_{p}$ error rate of order at least $1/2-$ for $p \geq 1$ in this setting.
	Furthermore in \cite{MGY20} for the case $d = 1$ a Milstein-type method was introduced, which achieves
	an $L_{p}$ error rate of order at least $3/4$. It was later proven in \cite{MGY21} that the error rate $3/4$ cannot be improved by any method
	based on finitely many evaluations of $W$ on a fixed grid. 
	
	In this article, we introduce, for the first time in literature, a Milstein-type method which can be used to approximate SDEs of this type for general $d \in \N$ and prove 
	that this Milstein-type scheme achieves an $L_{p}$-error rate of order at least $3/4-$ in terms of the number of steps. This method depends, in addition to evaluations of $W$ on a fixed grid, also on iterated integrals w.r.t. components of $W$, which can in general not be represented as functionals of $W$ evaluated at finitely many time points. We additionally prove that our suggested Milstein-type method  is only dependent on evaluations of $W$ on a finite, fixed grid if $\sigma$ is additionally commutative. 
	To obtain our main result we prove that a quasi-Milstein scheme achieves an $L_{p}$-error rate of order at least $3/4-$ in our setting if $\mu$ is additionally continuous, which is of interest in itself.
\end{abstract}

\maketitle

	\section{Introduction}
	Let 
	$ ( \Omega, \mathcal{F}, \PP ) $ 
	be a complete probability space with a 
	filtration $ ( \mathcal{F}_t )_{ t \in [0,1] } $
	that satisfies the usual conditions, let $d\in\N$ and 
	consider a $d$-dimensional autonomous stochastic differential equation (SDE)
	\begin{equation} \label{SDE}
		\begin{aligned}
			dX_{t} &= \mu(X_{t})dt + \sigma(X_{t})dW_{t}, \quad t \in [0,1], \\
			X_{0} &= x_{0},
		\end{aligned}
	\end{equation}
	where $x_0\in\R^d$,   $\mu\colon\R^d\to\R^d$ and   $\sigma\colon \R^d\to\R^{d \times d}$ are measurable functions and\\ $
	W=(W_{1},\dots,W_{d}) \colon [0,1] \times \Omega \to \R^d
	$
	is a $d$-dimensional
	$ ( \mathcal{F}_t )_{ t \in [0,1] } $-Brownian motion
	on $ ( \Omega, \mathcal{F}, \PP ) $.
	
	It is well-known that if the coefficients $\mu$ and $\sigma$ are Lipschitz continuous then the SDE \eqref{SDE} has a unique strong solution $X$. 
	Moreover, provided $\mu$ and $\sigma$ are $C^{1}$ the Milstein scheme is given recursively by 
	$\eul_{n,0}=x_0$ and
	\begin{align*}
	\eul_{n,(i+1)/n}=&\eul_{n,i/n}+\mu(\eul_{n,i/n})\, 1/n+\sigma(\eul_{n,i/n})\, (W_{(i+1)/n}-W_{i/n}) \\
	&+ \sum_{j_{1},j_{2} = 1}^{d} ((\sigma_{j_{2}})' \cdot \sigma_{j_{1}})(\eul_{n,i/n}) \int_{i/n}^{(i+1)/n} W_{j_{1},s} - W_{j_{1},i/n} \, dW_{j_{2},s} 
	\end{align*}
	for $i\in\{0,\ldots,n-1\}$, where for $x \in \R^{d}$ and $j \in \{1,\dots,d\}$ we write $\sigma_{j}(x)$ for the $j$-th column of $\sigma(x)$.
	It is well known, that the iterated Ito integrals  $\int_{i/n}^{(i+1)/n} W_{j_{1},s} - W_{j_{1},i/n} \, dW_{j_{2},s}$ for $j_{1},j_{2} \in \{1,\dots,d \}$ can in general not be represented as functionals of $W_{1/n},W_{2/n},\dots,W_{1}$, making the Milstein method dependent on access to these iterated integrals. However, if $\sigma$ fulfils the commutativity condition
	\[
		(\sigma_{j_{1}})'(x)\sigma_{j_{2}}(x) = (\sigma_{j_{2}})'(x)\sigma_{j_{1}}(x) \text{ for } x \in \R^{d} \text{ and } j_{1},j_{2} \in \{1,\dots,d \},
	\]
	the Milstein-scheme is based only on $W_{1/n},W_{2/n},\dots,W_1$. We note that this condition is automatically fulfilled if $d = 1$ holds.
	It is known, that provided $\mu$ and $\sigma$ are $C^{1}$ with bounded, Lipschitz continuous derivative, the Milstein scheme achieves at the final time $1$ an $L_p$-error rate of at least $1$ for all $p\geq 1$ in terms of the number $n$ of steps, i.e.	for all $p\geq 1$  there exists $c>0$ such that for all $n\in\N$,
	\begin{equation}\label{Leul}
		\bigl(\EE\bigl[\|X_1-\eul_{n,1}\|^p\bigr]\bigr)^{1/p}\leq \frac{c}{n},
	\end{equation}
	where $\|x\|$ denotes  the Euclidean norm of  $x\in\R^d$.
	
	In this article we study the performance of Milstein-type schemes in the case when the drift coefficient $\mu$ is discontinuous. Such SDEs  arise e.g.  in insurance, mathematical finance  and stochastic control problems, see e.g.~\cite{B80,Ka11,SS16} for examples.
	
	We assume that
	the drift coefficient  $\mu$ is piecewise Lipschitz continuous with an exceptional set $\Theta$ which is an orientable $C^5$-hypersurface of positive reach. The diffusion coefficient $\sigma$ is assumed to be Lipschitz continuous. Furthermore it is assumed that $\mu$ and $\sigma$ are $C^{1}$ on $\R^{d}\setminus \Theta$ with intrinsic Lipschitz continuous derivative. Moreover, in a  neighborhood of $\Theta$, $\mu$ and $\sigma$ are bounded and $\sigma$ is non-degenerate. See conditions (A) and (B) in Section \ref{ErrEs} for the precise assumptions on $\mu$ and $\sigma$.
	
	For such SDEs, existence and uniquness of a strong solution was recently proven in \cite{LS17}.
	Moreover, in \cite{LS17, LS18, NSS19}  $L_2$-appoximation of $X_1$ was studied. More precisely, in \cite{LS17} 
	an  $L_2$-error rate of at least $1/2$ was shown for a transformation-based Euler-Maruyama scheme.	This scheme is obtained by first applying a suitable transformation to the SDE \eqref{SDE} to obtain an SDE with Lipschitz continuous coefficients, then using the Euler-Maruyama scheme to approximate the solution of the transformed SDE and finally  applying the inverse of the transformation to the Euler-Maruyama scheme  to obtain an approximation of $X_1$.
	In \cite{ NSS19} an adaptive Euler-Maruyama scheme was constructed that
	adapts its step size to the actual distance of the scheme to the exceptional set $\Theta$ of $\mu$ -- it uses smaller  time steps the smaller the distance to $\Theta$ is. This scheme was shown to achieve an $L_2$-error rate of at least $1/2-$ (i.e, $1/2-\delta$ for every $\delta>0$) in terms of the average number of evaluations of $W$. Furthermore in \cite{SN21} it was proven that the Euler-Maruyama scheme achieves an $L_{p}$ error rate of at least $3/4-$ for a class of SDEs with irregular (possibly discontinuous) drift coefficient and additive diffusion.  

	In  \cite{LS18} the performance of the  Euler-Maruyama scheme $\eul_{n,1}$ for such SDEs was studied  and 
	an $L_2$-error rate of at least $1/4-$ was proven if the coefficients $\mu$ and $\sigma$ are additionally globally bounded. Recently this rate was improved in \cite{MGY20} for the case $d = 1$ where it was proven that the Euler-Maruyama scheme achieves an $L_{p}$ error rate of order at least $1/2$ in the case $d = 1$. This result was then further generalized very recently in \cite{MGYR24}, where an $L_{p}$-error rate of order at least $1/2-$ was proven for arbitrary dimensions. It is worth noting at this point that all these results for Euler-Maruyama-type methods hold under weaker assumptions than the ones stated above. However to the best of our knowledge the recent seminal article \cite{MGY19b}, covering the case $d = 1$, is the first result suggesting a higher order method based on evaluations of the underlying Brownian motion on a finite fixed grid under these assumptions. There a transformed Milstein-scheme is considered for which an $L_{p}$ error rate of order at least $3/4$ in terms of the number of evaluations of the Brownian motion was proven. This rate was later shown to be optimal among all algorithms for the approximation of $X_{1}$ based on $n$ evaluations of $W$ at fixed time points, see \cite{Ell24, MGY21} for  matching lower error bounds.
	
	In the present article we present the first higher order scheme under these assumptions which works for general $d \in \N$.
	More precisely, we show that there is a method $\eul_{n}$ based on $W_{1/n},W_{2/n},\dots,W_{1}$ and $\int_{i/n}^{(i+1)/n} W_{j_{1},s} - W_{j_{1},i/n} \, dW_{j_{2},s}$, $j_{1},j_{2} \in \{1,\dots,d \}$, $i \in \{1,\dots,n-1 \}$ such that for all $p\geq 1$ and all $\delta>0$ there exists $c>0$ such that for all $n\in\N$, 
	\begin{equation}\label{UBE}
		\bigl(\EE\bigl[\|X_1-\eul_{n,1}\|^p\bigr]\bigr)^{1/p}\leq \frac{c}{n^{3/4-\delta}},
	\end{equation}
	i.e., the suggested scheme $\eul_{n,1}$ achieves an $L_p$-error rate of at least $3/4-$ in terms of the number of steps for all $p\geq 1$. The scheme used here is a transformed Milstein-scheme which turns out to only be dependent on $W_{1/n},W_{2/n},\dots,W_{1}$ if the diffusion coefficient is additionally commutative.
	This upper bound follows directly from our main result, Theorem \ref{Thm1}, which states that for all $p\geq 1$ and all $\delta>0$ the supremum error of a time continuous version of the scheme achieves the rate of at least  $3/4-\delta$ in the $L_p$-sense, see Section \ref{ErrEs}.
	
	In order to achieve this result, two central hurdles had to be overcome. The first hurdle is that of constructing a suitable transformation mapping which can be used to transform the problem with discontinuous drift coefficient to a problem with Lipschitz continuous coefficients such that the $L_{p}$ error rate of the Milstein-scheme itself could be analyzed for the transformed problem. This transformation is similar to the one used in \cite{LS17}, however the transformation there had to be modified in order for the transformed coefficients to fulfil additional regularity assumptions required for the analysis of the Milstein-scheme.
	The second central hurdle that had to be overcome in order to prove Theorem \ref{Thm1} was the analysis of the Milstein-scheme in the transformed setting. The coefficients in this setting turn out to be Lipschitz continuous and $C^{1}$ on $\R^{d} \setminus \Theta$ with intrinsic Lipschitz continuous derivative there. Here the use of the term "Milstein scheme" is slightly imprecise, as the classical Milstein scheme requires the coefficients to be globally $C^{1}$. However the scheme used here is simply the classical Milstein scheme with the derivative of $\sigma$ set to zero on $\Theta$, hence the name is still justified we think. For the analysis of this scheme, in order to overcome the lack in regularity in this situation, a suitable bound on expectations of the form $\EE\bigl(\|f(\eul_{n,t}) - f(\eul_{n,\underline{t}_{n}})\|\bigr)$, where $f \colon \R \to \R$ is a piecewise Lipschitz continuous function and $\underline{t}_{n} = \frac{\floor{tn}}{n}$ had to be found. This was done by analysing $\EE\bigl(\|f(\eul_{n,t}) - f(\eul_{n,\underline{t}_{n}})\|\one_{A}\bigr)$, where $A$ is the set of all $\omega \in \Omega$ for which $\eul_{n,t}(\omega)$ and $\eul_{n,\underline{t}_{n}}(\omega)$ lie on "different sides" of $\Theta$. This term is problematic, as the piecewise Lipschitz continuity of $f$ does not provide a suitable bound on $\| f(\eul_{n,t}(\omega)) - f(\eul_{n,\underline{t}_{n}}(\omega)) \|$ directly. The proof of Theorem \ref{Thm1} is thus based on a detailed analysis of the expected time $\eul_{n,\underline{t}_{n}}$ and $\eul_{n,t}$ spend on different sides of the hypersurface $\Theta$. This involves estimating the time the process $\eul_{n}$ spends in a neighborhood of $\Theta$, which is achieved by transforming the problem to a one dimensional one and then utilising occupation time arguments. This analysis of the Milstein scheme in the transformed setting represents a strong error result for the Milstein-scheme under nonstandard assumptions, and thus has some significance on its own. Hence this result and the precise assumptions needed for it to hold, see Assumptions (C) and (D) are stated in a separate Theorem \ref{Thm2}.
	
	We add that for $d$-dimensional SDEs \eqref{SDE} with a discontinuous drift coefficient  in a different setting than the one considered in this article an $L_p$-error rate of at least $1/2-$ for all $p\geq 1$  for the Euler-Maruyama scheme $\eul_{n,1}$ was recently independently proven in \cite{DGL22}. More precisely, in \cite{DGL22} 
	the drift coefficient $\mu$ is assumed to be bounded and measurable and the diffusion coefficient $\sigma$ is assumed to be bounded, uniformly elliptic and twice continuously differentiable with bounded partial derivatives of order
	$1$ and $2$. In particular, $\sigma$ is non-degenerate everywhere. Moreover, for additive noise driven SDEs \eqref{SDE} with  $\sigma$ equal to the identity matrix an $L_p$-error rate of at least  $1/(2\max(2,d,p))+1/2-$ for all $p\geq 1$ was shown in \cite{DGL22}  for the Euler-Maruyama scheme $\eul_{n,1}$  for discontinuous drift coefficients $\mu$ of the form $
	\mu=\sum_{i=1}^n f_i \one_{K_i}$ with bounded
	Lipschitz domains $K_1, \ldots, K_n\subset \R^d$ and bounded Lipschitz continuous functions $f_1, \ldots, f_n\colon\R^d\to\R^d$.
	
	We also refer to
	\cite{MGY24} for a recent survey on the complexity of $L_p$-approximation of one-dimensional SDEs with a discontinuous drift coefficient. Furthermore in \cite{GLL24}, the Milstein scheme was analysed for SDEs with a drift coefficient which is only Hölder-continuous. Furthermore we refer to \cite{PSS24} for the construction and analysis of a transformed Milstein scheme for the approximation of one dimensional jump diffusion SDEs in a similar setting to ours and to \cite{PSS25} for matching lower error bounds.
	Furthermore we note, that recently in \cite{Y21} a Milstein-type method based on sequential evaluations of the underlying Brownian Motion was constructed and it was shown that this method achieves an $L_{p}$ error rate of at least $1$ based on $n$ sequentially chosen evaluations of $W$ on average. An $L_p$-error rate better  than $1$  can not be achieved in general  by no numerical method based on $n$ sequentially chosen evaluations of $W$ on average, see \cite{hhmg2019, m04} for matching lower error bounds. For jump-diffusion SDEs we refer to the very recent article \cite{S25}, where a scheme based on sequential evaluations of the underlying Brownian Motion was analysed in a similar setting.
	The extension of the upper bounds from \cite{Y21}  to the class of $d$-dimensional SDEs considered in the present article 
	using techniques developed in the article will be  the subject of future work.

	We briefly describe the content of the paper. Our error estimates, Theorem~\ref{Thm1} and Theorem~\ref{Thm2}, are stated in Section~\ref{ErrEs}. Section~\ref{Proofs} contains 
	proofs of these results. In Section~\ref{Exmpl} we present some examples. Section~\ref{Num} is devoted to numerical experiments. In Section~\ref{Appendix} we state some known results from differential geometry, Lemmas \ref{connect0} to \ref{a0}, that are used for our proofs in Section~\ref{Proofs}.
\section{Error estimates for the Milstein scheme} \label{ErrEs}

We briefly recall the notions of tangent vector, normal vector, orthogonal projection and reach from differential geometry as well as the notion of piecewise Lipschitz continuity that has been introduced in~\cite{LS17} and employed in~\cite{LS18} for the analysis of the Euler scheme.

Let $\emptyset\neq \Theta \subset \R^d$. If $x\in \Theta$ then $v\in \R^d$ is called a tangent vector to $\Theta$ at $x$ if there exist $\epsilon\in (0,\infty)$ and a $C^1$-mapping $\gamma\colon (-\epsilon,\epsilon) \to \Theta$ such that $\gamma(0)=x\text{ and }\gamma'(0) = v$.

A function $\nor\colon \Theta \to \R^{d}$ is called normal vector along $\Theta$ if $\nor$ is continuous with $\|\nor\| =1$ and  $\langle \nor(x),v\rangle = 0$ for every $x\in\Theta$ and every tangent vector $v$ to $\Theta$ at $x$, where $\|\cdot\|$ and $\langle\cdot, \cdot\rangle$ denote the euclidean norm and the euclidean scalar product, respectively.
The set $\Theta$ is called orientable if there exists a normal vector along $\Theta$.

The Lipschitz continuous mapping
\[
d(\cdot, \Theta)\colon\R^d\to [0,\infty),\,\, x\mapsto \inf\{\|y-x\|\colon y\in \Theta\}
\]
is called the distance function of $\Theta$. The set
\[
\unp(\Theta) = \{x\in \R^d\colon \exists_1 y\in\Theta\colon \|y-x\|=d(x,\Theta)\}
\]
consists of all points in $\R^d$ that have a unique nearest point in $\Theta$ and the mapping
\[
\pr_\Theta\colon \unp(\Theta)\to \Theta,\,\, x\mapsto \text{argmin}_{y\in \Theta} \|x-y\|
\]
is called the orthogonal projection onto $\Theta$.  
For  $\eps\in [0,\infty)$, the $\eps$-neighbourhood of $\Theta$ is given by the open set
\[
\Theta^\eps = \{x\in\R^d\colon d(x,\Theta)< \eps \},
\]
and the quantity
\[
\reach(\Theta) = \sup\{\eps\in [0,\infty)\colon \Theta^\eps\subset \unp(\Theta)\} \in [0,\infty]
\]
is called the reach of $\Theta$.

Next, recall that the length of a continuous function $\gamma\colon  [0,1] \to \R^{d}$ is defined by
\[
l(\gamma) = \sup \left\{ \sum_{k = 1}^{n} \|\gamma(t_{k}) - \gamma(t_{k-1})\| \colon 0 \leq t_{0} < \dots < t_{n} \leq 1,\, n\in \N   \right\} \in [0,\infty]
\]
and that for a nonempty subset $A$ of $\R^d$ the intrinsic metric $\rho_A\colon A\times A \to [0,\infty]$ is given by
\[ 
\rho_A(x,y) = \inf\{l(\gamma) \colon \gamma\colon  [0,1] \to A \text{ is continuous with } 
\gamma(0) = x \text{ and } \gamma(1) = y \},\quad x,y\in A.
\]
Let $\emptyset\neq A\subset D\subset \R^d$ and $m,k\in\N$. 
A  function $f\colon  D \to \R^{k\times m}$ is called intrinsic Lipschitz continuous on $A$, if there exists $L \in (0,\infty)$ such that  for all $x,y \in A$ we have $\|f(x) - f(y)\| \leq L \rho_A(x,y)$. In this case, $L$ is called an intrinsic Lipschitz constant for $f$ on $A$. If $A=D$ then $f$ is called intrinsic Lipschitz continuous.

A function $f\colon  \R^{d} \to \R^{k\times m}$  is called piecewise Lipschitz continuous if there exists a hypersurface $\Theta \subset \R^{d}$ such that $f$ is intrinsic Lipschitz continuous on $\R^d\setminus\Theta$. In this case, the hypersurface $\Theta$ is called exceptional set for $f$.

The aim of this paper is to show that a transformed Milstein scheme achieves a convergence rate of order at least $3/4-$ if  
the drift coefficient $\mu$ and the diffusion coefficient $\sigma$ satisfy the following conditions.
\begin{itemize}
	\item [(A)] There exists an orientable $C^{5}$-hypersurface $\Theta\subset \R^d$ of positive reach
	such that\\[-.3cm]
	\begin{itemize}
		\item[(i)] there exists  a normal vector $\nor$ along $\Theta$ and an open neighborhood $U\subset \R^d$ of $\Theta$ such that $\nor$ can be extended to a $C^4$-function $\nor\colon U\to\R^d$ that has bounded partial derivatives up to order $4$ on $\Theta$,\\[-.3cm]
		\item[(ii)] we have $\inf_{x\in \Theta} \|\normal(x)^\top \sigma(x)\| > 0$,\\[-.3cm]
		\item[(iii)] there exists   an open neighborhood $U\subset \R^d$ of $\Theta$ such that the function
		\[
		\alpha\colon \Theta\to\R^d,\,\, 	x\mapsto \lim_{h \downarrow 0}\frac{\mu(x-h\normal(x))- \mu(x+h\normal(x))}{2 \|\sigma(x)^\top\normal(x)\|^2}\\[.1cm]
		\]
		can be extended to a $C^4$-function $\alpha\colon U\to\R^d$ that has bounded partial derivatives up to order $4$ on $\Theta$,\\[-.3cm]
		\item[(iv)] $\mu_{|\R^{d}\setminus \Theta}$ and $\sigma_{|\R^{d}\setminus \Theta}$ are $C^{1}$.
		\item[(v)] there exists $\eps\in (0,\reach(\Theta))$ such that $\mu$ and $\sigma$ are bounded on $\Theta^{\eps}$,\\[-.3cm]
		\item[(vi)] the functions $\mu_{|\R^{d}\setminus \Theta}$, $(\mu_{|\R^{d}\setminus \Theta})'$ and $((\sigma_{j})_{|\R^{d}\setminus \Theta})'$ for $j \in \{1,\dots,d \}$ are intrinsic Lipschitz continuous.\\[-.3cm]
	\end{itemize}
\item[(B)] $\sigma$ is Lipschitz continuous.
\end{itemize}
For later use we add the following Lemma about Assumptions (A) and (B).
\begin{Lem} \label{propimp} Assume that $\mu$ and $\sigma$ fulfil assumptions (A)(ii),(iv),(vi) and B. Then the following hold
	\begin{itemize}
		\item[i)] There exists $K \in (0,\infty)$ such that for all $x \in \R$ we have
		\[
		\norm{\mu(x)} + \norm{\sigma(x)} \leq K (1+\norm{x})
		\]
		\item[ii)] The functions $\mu'$ and $\sigma'$ are bounded on $\R^{d}\setminus \Theta$.
		\item[iii)]  There exists $\epsilon \in (0,\infty)$ such that we have 
	\begin{align} \label{remcond0}
	\inf_{x\in \Theta^{\epsilon}} \|\normal(pr_{\Theta}(x))^\top \sigma(x)\| > 0
	\end{align}
	\end{itemize}
\end{Lem}
\begin{proof}
	For part i) we refer to \cite[Lemma 1]{MGYR24}. For part (ii) we note, that the boundedness of $\sigma'$ on $\R^{d}\setminus\Theta$ follows immediately from (A)(iv) and (B). To deduce the boundedness of $\mu$ on $\R^{d}\setminus \Theta$ from (A)(iv) and (vi) let $L \in (0,\infty)$ be an intrinsic Lipschitz constant for $\mu$. As $\R^{d}\setminus \Theta$ is open, for $x \in \R^{d}\setminus \Theta$ there exists $\epsilon > 0$ such that $B_{\epsilon}(x) \subseteq \R^{d}\setminus \Theta$. As $B_{\epsilon}(x)$ is convex, the function $\mu_{|B_{\epsilon}(x)}$ is Lipschitz continuous with Lipschitz constant $L$ by (A)(vi) and Lemma \ref{lipimpl0}(ii). Thus $\norm{\frac{\partial \mu}{\partial x_{i}}} \leq L$ holds for $i \in \{1,\dots,d \}$.
	To prove part iii), put $\epsilon :=  \frac{\inf_{x\in \Theta} \|\normal(x)^\top \sigma(x)\|}{L+1}$, where $L$ is a Lipschitz constant for $\sigma$. Then for $x \in \Theta^{\epsilon}$ we have
	\begin{align*}
		\| \normal(\pr_{\Theta}(x))\sigma(x) \| &\geq \| \normal(\pr_{\Theta}(x))\sigma(\pr_{\Theta}(x)) \| - \| \normal(\pr_{\Theta}(x))(\sigma(x) - \sigma(\pr_{\Theta}(x))) \| \\
		&\geq (L+1)\epsilon - \| \sigma(\pr_{\Theta}(x)) - \sigma(x) \| \geq (L+1) \epsilon - L\| \pr_{\Theta}(x) - x \| \geq \epsilon.
	\end{align*}
	
\end{proof}
We briefly recall a well-known fact from differential geometry. 

Let $\Theta\subset\R^d$ be a $C^1$-hypersurface  of positive reach and a normal vector $\normal\colon \Theta\to\R^d$.
For $s\in\{+,-\}$ and $\eps\in (0,\reach(\Theta))$  put
\begin{equation}\label{u00}
	Q_{\eps,s} = \{x+s\lambda\normal(x)\colon x\in\Theta, \lambda\in (0,\eps) \}.
\end{equation}
Since $\Theta$ is an orientable $C^1$-hypersurface of positive reach it follows that $Q_{\eps,+}$ and $Q_{\eps,-}$ are open and disjoint with 
\begin{equation}\label{u01}
	\Theta^\eps \setminus \Theta = Q_{\eps,+}\cup Q_{\eps,-},
\end{equation}
see \app{Lemma~\ref{sides0}}. 

Using~\eqref{u01} we can prove the existence of the limit on the right hand side in condition (A)(iii), see Remark~\ref{remdisc}. For a proof we refer to \cite[Lemma 2]{MGYR24}.

\begin{Lem}\label{exlim} Assume that $\mu$  satisfies (A)(vi). Then for every $s\in\{+,-\}$ and every $x\in\Theta$, the limit 
	\[
	\lim_{h\downarrow 0} \mu (x+sh\normal(x))
	\]
	exists in $\R^d$.
\end{Lem}
As mentioned above, these conditions (A) and (B) are very similar to conditions (A) and (B) from \cite{MGYR24}, with added regularity assumptions needed for the analysis of the transformed Milstein scheme.
We refer to Remarks 1, 2 and 3 from \cite{MGYR24} for motivation and interpretations of the quantities mentioned in these assumptions. We add the following remark, which is similar to Remark 4 in \cite{MGYR24}.
\begin{Rem}\label{remdisc}
	In the one-dimensional case, i.e. $d=1$, it is easy to check that the conditions (A) and (B) are equivalent to the following four conditions:
	\begin{itemize}
		\item[(i)] $\Theta \subset \R$ is countable with $\delta:=\inf\{|x-y|\colon x,y\in \Theta, x\neq y\} > 0$,
		\item[(ii)] $\sigma\colon\R\to\R$ is Lipschitz continuous with
		\[
		0 < \inf_{x\in \Theta} |\sigma(x)| \le  \sup_{x\in \Theta} |\sigma(x)| < \infty,
		\]
		\item[(iii)] $\mu\colon\R\to\R$ and $\sigma \colon \R \to \R$ are $C^{1}$ on each of the countable intervals $(x,y) \subset \R$ with $x,y\in \Theta\cup\{-\infty,\infty\}$ and $(x,y)\cap \Theta = \emptyset$
		\item[(iv)] $\mu$, $\mu'$ and $\sigma'$ are Lipschitz continuous on each of the countable intervals $(x,y) \subset \R$ with $x,y\in \Theta\cup\{-\infty,\infty\}$ and $(x,y)\cap \Theta = \emptyset$ and for every $x\in\Theta$  there exist $y_x,z_x\in \R$ with $x-\delta < y_x < x< z_x < x+\delta$ such  that
		\[
		\sup_{x\in \Theta} ( |\mu(x)| + |\mu(y_x)| + |\mu(z_x)|) < \infty.
		\]
		
	\end{itemize}
	A particular instance where (i)-(iv) are fulfilled is given by $\Theta = \{x_1,\dots,x_K\}$ with $-\infty = x_0 < x_1 < \dots < x_K < x_{K+1} = \infty$ such that $\sigma$ is Lipschitz continuous with $\sigma(x_k)\neq 0 $ for every $k\in \{1,\dots,K\}$ and $\mu$,$\mu'$ as well as $\sigma'$ exist and are Lipschitz continuous on $(x_k,x_{k+1})$ for every $k\in \{0,\dots,K\}$. The latter setting is studied in \cite{MGY19b}.
\end{Rem}
We note, that it was proven in \cite{LS17}, that under weaker assumptions than (A) and (B) a unique strong solution of \eqref{SDE} exists.

In order to analyse the transformed Milstein scheme which will be introduced later, we will first prove that under the following assumptions the Milstein scheme itself behaves as desired.
\begin{itemize}
	\item[(C)] There exist $C^{3}$-hypersurfaces $\Theta, \Delta \subset \R^d$ of positive reach such that
	\begin{itemize}
		\item[(i)] For $\Gamma \in \{\Theta,\Delta \}$ there exists a normal vector $\nor$ along $\Gamma$ such that we have \\
		$\inf_{x\in \Gamma} \|\nor(x)^\top \sigma(x)\| > 0$,
		\item[(ii)] the functions $\mu_{|\R^{d}\setminus \Theta}$ and $\sigma_{|\R^{d}\setminus \Delta}$ are $C^{1}$,
		\item[(iii)] there exists $\epsilon\in (0,\infty)$ such that $\sigma$ is bounded on $\Theta^{\epsilon}$ and on $\Delta^{\epsilon}$,
		\item[(iv)] the functions $(\mu_{|\R^{d}\setminus \Theta})'$ and $((\sigma_{j})_{|\R^{d}\setminus \Delta})'$ for $j \in \{1,\dots,d \}$ are intrinsic Lipschitz continuous.
	\end{itemize}
	\item[(D)] $\mu$ and $\sigma$ are Lipschitz continuous.
\end{itemize}
We turn to our results. Assume that there exists a $C^{3}$-hypersurface $\Delta$ of positive reach such that $\sigma$ is $C^{1}$ on $\R^{d}\setminus \Delta$.
For $n\in\N$ let $\eul_{n}=(\eul_{n,t})_{t\in[0,1]}$  denote  the time-continuous Milstein scheme with step-size $1/n$ associated to the SDE \eqref{SDE}, i.e. $\eul_{n}$ is recursively given by
$\eul_{n,0}=x_0$ and
\begin{equation} \label{MilSc}
\eul_{n,t}=\eul_{n,i/n}+\mu(\eul_{n,i/n})\, (t-i/n)+\sigma(\eul_{n,i/n})\, (W_t-W_{i/n}) + \sum_{j_{1},j_{2} = 1}^{d} (\partial\sigma_{j_{2}} \cdot \sigma_{j_{1}})(\eul_{n,\underline{t}_{n}}) J_{j_{1},j_{2}}^{n}(t) 
\end{equation}
for $t\in (i/n,(i+1)/n]$ and $i\in\{0,\ldots,n-1\}$,
where for $x \in \R^{d}$ and $j \in \{1,\dots,d\}$ we write $\sigma_{j}(x)$ for the $j$-th column of $\sigma(x)$, where 
\begin{align} \label{diffdef}
	\partial\sigma_{j}(x) = \begin{cases}
		 \sigma_{j}'(x), & \text{ if } x \in \R^{d}\setminus \Delta \\
		 0, & \text{otherwise}.
	\end{cases}
\end{align}
with $\sigma_{j}'$ the derivative of $\sigma_{j}$ and where for $j_{1},j_{2} \in \{1,\dots,d \}$ and $n \in \N$ we define  
\[
	J_{j_{1},j_{2}}^{n}(t) := \int_{\underline{t}_{n}}^{t} W_{j_{1},s} - W_{j_{1},\underline{t}_{n}} \, dW_{j_{2},s}
\]

We have the following error estimate for $\eul_{n}$ under assumptions (C) and (D)

\begin{Thm}\label{Thm2} Assume that $\mu$ and $\sigma$ satisfy $(C)$ and $(D)$.
	For every $p\in [1,\infty)$ and every $\delta\in (0,\infty)$ there exists $c\in(0, \infty)$ such that for all $n\in\N$, 
	\begin{equation}\label{l3}
		\bigl(\EE\bigl[\|X-\eul_{n}\|_\infty^p\bigr]\bigr)^{1/p}\leq \frac{c}{n^{3/4-\delta}}. 
	\end{equation}
	If additionally $\Delta = \emptyset$, then there exists $c\in(0, \infty)$ such that for all $n\in\N$, 
	\begin{equation}\label{l4}
		\bigl(\EE\bigl[\|X-\eul_{n}\|_\infty^p\bigr]\bigr)^{1/p}\leq \frac{c}{n^{1-\delta}}. 
	\end{equation}
\end{Thm}

With the help of the aforementioned transform we will show that the following theorem is a consequence of Theorem \ref{Thm2}.

\begin{Thm}\label{Thm1} Assume that $\mu$ and $\sigma$ satisfy $(A)$ and $(B)$. Then there exists a sequence of measurable functions $\phi_{n}\colon (\R^{d})^{n}\times(\R^{d \times d})^{n} \to \R^{d}$ such that
	for every $p\in [1,\infty)$ and every $\delta\in (0,\infty)$ there exists $c\in(0, \infty)$ such that for all $n\in\N$, 
	\begin{equation*}
		\bigl(\EE\bigl[\|X_{1}-\phi_{n}(W_{\frac{1}{n}},W_{\frac{2}{n}},\dots,W_{1},J^{n}(0),\dots,J^{n}(n-1))\|^p\bigr]\bigr)^{1/p}\leq \frac{c}{n^{3/4-\delta}}. 
	\end{equation*}
	where for $k \in \{0,\dots,n-1 \}$
	\[
	J^{n}(k) = [J_{i,j}^{n}(k)]_{i,j = 1}^{d}.
	\] 
	If $\sigma$ additionally fulfils the commutativity condition
	\[
	(\sigma_{j_{1}})'(x)\sigma_{j_{2}}(x) = (\sigma_{j_{2}})'(x)\sigma_{j_{1}}(x) \text{ for } x \in \R^{d}\setminus\Theta \text{ and } j_{1},j_{2} \in \{1,\dots,d \}
	\]
	then there exists a sequence of measurable functions $\psi_{n}\colon (\R^{d})^{n} \to \R^{d}$ such that
	for every $p\in [1,\infty)$ and every $\delta\in (0,\infty)$ there exists $c\in(0, \infty)$ such that for all $n\in\N$, 
	\begin{equation*}
		\bigl(\EE\bigl[\|X_{1}-\psi_{n}(W_{\frac{1}{n}},W_{\frac{2}{n}},\dots,W_{1})\|^p\bigr]\bigr)^{1/p}\leq \frac{c}{n^{3/4-\delta}}. 
	\end{equation*}
\end{Thm}

\section{Proofs} \label{Proofs}

We briefly describe the structure of the proof.
In order to prove Theorem \ref{Thm1}, a transformation mapping is introduced in Subsection \ref{trans}, which is used to
transform the SDE \eqref{SDE} under Assumptions (A) and (B) to an SDE with coefficients
which fulfil Assumptions (C) and (D). The transformation used here differs from the one used
in \cite{LS17}, hence this subsection contains the proofs of a variety of properties of the transformation
which are used in the text. The subsequent subsections are then dedicated to analysing the
time continuous Milstein-scheme in the transformed setting, i.e. a setting with coefficients which fulfil Assumptions (C) and (D). First some properties of the coefficients have to be
deduced from the assumptions (C) and (D), which is done in Subsection \ref{propcoeff}.
In Subsection \ref{occtimechap}, necessary moment estimates
as well as a Markov-property for the time continuous Milstein-scheme $\eul_{n}$ are provided. Furthermore the time the process
$\eul_{n}$ spends in a neighborhood of $\Theta$ is analysed.
All this is then put to use when it is shown that under Assumptions (C) and (D), for all $p \in [1,\infty)$ there
exists a constant $c \in (0,\infty)$ such that 
\begin{equation} 
	\EE\biggl[\Bigl|	\int_{0}^{1} \one_{\{d(\eul_{n,t}, \Gamma)\le\|\eul_{n,t} - \eul_{n,\utn}\| \}}\, dt\Bigr|^p\biggr]^{1/p} \leq \frac{c}{n^{\frac{1}{2}-\delta}}.
\end{equation}
for all $n \in \N$ and $\Gamma \in \{\Theta,\Delta \}$.
This is a crucial tool in the analysis of the Milstein scheme in this setting.

\subsection{Notation}
Let $d,m\in\N$. 

For a matrix $A\in\R^{d\times m}$ we use $\|A\|$ to denote the Frobenius norm of $A$, $A_{j}$ to denote the $j$-th column of $A$ for $j \in \{1,\dots,m\}$, $A^{i}$ to denote the $i$-th row of $A$ for $i \in \{1,\dots,d\}$.
In the case $ d=m$ we use $\tr(A)$ to denote the trace of $A$ and $\text{det}(A)$ to denote the determinant of $A$. 
For $x,y\in\R^d$ we use  $\langle x,y\rangle$ to denote the Euclidean scalar product of $x$ and $y$ and
$ \overline{x,y} = \{\lambda x + (1-\lambda)y \colon  \lambda \in [0,1] \}
\subset \R^d
$ to denote the straight line connecting  $x$ and $y$. For $r\in [0,\infty)$ and $x\in \R^d$ we use $B_{r}(x) = \{y \in \R^{d} \colon \|x-y\| < r \}$ to denote the open ball and $\overline{B}_{r}(x) = \{y \in \R^{d} \colon \|x-y\| \leq r \}$ to denote the closed ball with center $x$ and radius $r$, respectively.

For a set $U \subset \R^{d}$ we write $\inter(U)$, $\cl(U)$ and $\partial U$ for the interior, the closure and the boundary of $U$, respectively. For  a function $f \colon U \to \R^{m}$ and a set 
$ M\subset U$ 
we use $\|f\|_{\infty,M} = \sup\{\|f(x)\|\colon x\in M\}$ to denote the supremum of 
 the values of 
$\|f\|$ on $M$.

For a multi-index $\alpha=(\alpha_1,\dots,\alpha_d)
^\top
\in \N_0^d$ we put $|\alpha| = \alpha_1+\dots + \alpha_d$.
For a set $ U\subset \R^d$, an open set  $\emptyset\neq M\subset U$, $k\in\N_0$ and a function $f=(f_1,\dots,f_m)^\top \colon U \to \R^{m}$, which is a $C^{k}$-function on $M$, we put
\[ 
\|f^{(\ell)}(x)\|_\ell = \max_{i\in\{1,\dots,m\}}\max_{\alpha \in \N_0^{d}, \abs{\alpha} = \ell} |f_i^{(\alpha)}(x) |
\] 
for $\ell \in \{0,1,\dots,k \}$ and $x \in M$.  
If $k\ge 1$ then we use
\[
f'\colon M\to \R^{m\times d},\,\, x\mapsto \Bigl(\frac{\partial f_i}{\partial x_j}(x)\Bigr)_{\substack{1\le i\le m \\ 1\le j \le d}}\in\R^{m\times d}
\]
to denote the first derivative of $f$ on $M$. If $m=1$ and $k\ge 2$ then we use 
\[
f''\colon M\to \R^{d\times d},\,\, x\mapsto \Bigl(\frac{\partial^2 f}{\partial x_{j_1} \partial x_{j_2}}(x)\Bigr)_{1\le j_1,j_2\le d}\in\R^{d\times d}
\]
to denote the second derivative of $f$ on $M$.	

\subsection{The Transformation Procedure}\label{trans}
For our proofs, we rely on the existence of a transformation mapping which removes the discontinuity from the drift coefficient. This type of transform was first constructed by Gunther Leobacher and Michaela Szölgyenyi and was first introduced in \cite{LS17} and also used in \cite{MGYR24} for the analysis of the Euler-Maruyama scheme. It is easy to see that the conditions (A) and (B) from this text imply the assumptions from \cite{MGYR24} which are essentially the same as the ones in \cite{LS17}. Here a slightly modified version of the transform from \cite{LS17} is introduced. This modification is required to ensure that the transformed coefficients fulfil the regularity assumptions formulated in (C) and (D). Due to this modification, a proof of all the properties of the transform formulated in the following proposition are included in this text, even though assertions (i)-(v) and parts of (vi) of this proposition where already proven in \cite{MGYR24} for the transform introduced there and the proofs work very similarly here. We will refer to \cite{MGYR24} for proofs which transfer word by word.
\begin{Prop} \label{grep}
	Assume that $\mu$ and $\sigma$ satisfy (A) and (B) and let $\Theta\subset \R^d$ be an orientable $C^5$-hypersurface of positive reach  according to (A). Then there exists a function $G\colon\R^d\to\R^d$ with the following properties.
	\begin{itemize}
		\item[(i)] $ G$ is a $C^{1}$-diffeomorphism.
		\item[(ii)] $G$, $G^{-1}$, $G'$, $(G^{-1})' $ are Lipschitz continuous and $G'$, $(G^{-1})' $ are bounded.
		\item[(iii)] $G=(G_1,\dots,G_d)^\top$ and $G^{-1}=(G^{-1}_1,\dots,G^{-1}_d)^\top$ are $C^2$-functions on $\R^{d} \setminus \Theta$ with bounded, intrinsic Lipschitz continuous, second derivatives
		$G''_i, (G^{-1}_i)''\colon \R^{d} \setminus \Theta \to \R^{d\times d}$ for all $i\in\{1,\dots,d\}$.
		\item[(iv)] The function 
		\[
		\sigma_{G} =(G'\,\sigma)\circ G^{-1} \colon \R^{d} \to \R^{d\times d}
		\]
		is Lipschitz continuous with $	\sigma_G(x) = \sigma(x)$ for every $x\in\Theta$.
		\item[(v)] For every $i\in\{1,\dots,d\}$, the second derivative $G''_i \colon \R^{d} \setminus \Theta \to \R^{d\times d}$ of $G_{i}$ on $\R^d\setminus \Theta$ can be extended to a bounded mapping $ R_i\colon \R^d \to \R^{d\times d}$, such that the function
		\[
		\mu_G= \Bigl(G'\,\mu + \frac{1}{2} \bigl( \tr\bigl(R_i\,\sigma\sigma^{\top}\bigr))_{1\le i \le d}\Bigr)\circ G^{-1}\colon\R^d\to\R^d
		\]
		is Lipschitz continuous.
		\item[(vi)] $G_{|\R^{d}\setminus \Theta}$ is $C^{4}$.
		\item[(vii)] All partial derivatives of $G_{|\R^{d}\setminus \Theta}$ up to fourth order are bounded
		\item[(viii)] The functions $\mu_{G}$ and $\sigma_{G}$ are $C^{1}$ on $\R^{d}\setminus \Theta$ and $\mu_{G}'$ as well as $((\sigma_{G})_{i})'$ for $i \in \{1,\dots,d \}$ are intrinsic Lipschitz on $\R^{d} \setminus \Theta$
	\end{itemize}
\end{Prop}
With the help of this Proposition the following can be proven.
\begin{Prop} \label{transcomp}
	Assume $\mu$ and $\sigma$ fulfil the Assumptions (A) and (B). Then there exists a function $G\colon\R^d\to\R^d$ which fulfils (i)-(viii) from Proposition \ref{grep} and such that the functions $\mu_{G}$ and $\sigma_{G}$ from Proposition \ref{grep} fulfil the Assumptions (C) and (D)
\end{Prop}
\begin{proof} Choose $G$ from Proposition \ref{grep}.
	Recall that
	\[
	\sigma_{G} =(G'\,\sigma)\circ G^{-1} \colon \R^{d} \to \R^{d\times d}
	\]
	and
	\[
	\mu_G= \Bigl(G'\,\mu + \frac{1}{2} \bigl( \tr\bigl(R_i\,\sigma\sigma^{\top}\bigr))_{1\le i \le d}\Bigr)\circ G^{-1}\colon\R^d\to\R^d.
	\]
	Thus by Proposition \ref{grep}(iv)(v) and (viii) we already know that (C)(ii),(iv) and (D) are fulfilled with $\Theta = \Delta$.	
	Furthermore $\sigma$ and $G'$ are bounded on $\Theta^{\eps}$ and we have $G^{-1}(\Theta^{\eps}) = \Theta^{\eps}$ by A(v) and Proposition \ref{grep}(i),(ii). Thus C(iii) follows. 
	Finally note that by definition of $G$ we have $G_{|\Theta} = id_{|\Theta}$ and $G'_{|\Theta} = Id$. Hence we have
	\[
	\inf_{x \in \Theta} \norm{\nor(x)^{\top}\sigma_{G}(x)} = \inf_{x \in \Theta} \norm{(\nor(x))^{\top} (G' \, \sigma)(G^{-1}(x))} = \inf_{x \in \Theta} \norm{\nor(x)^{\top} \sigma(x)} > 0
	\]
	Thus C(i) is also fulfilled.

\end{proof}
For the proof of Proposition~\ref{grep}  we assume throughout the following that $\mu$ and $\sigma$ satisfy (A) and (B), we fix 
a $C^5$-hypersurface $\Theta\subset \R^d$ of positive reach 
and a normal vector $\nor$ along $\Theta$
according to (A), an open neighbourhood $U$ of $\Theta$ according to (A)(i),(iii) and $\eps^*\in (0,\reach(\Theta))$ such that (A)(v) holds with $\eps = \eps^*$.

First, we provide useful properties of the functions $\alpha$, $\pr_{\Theta}$, $\normal\circ\pr_{\Theta}$ and $\alpha\circ\pr_{\Theta}$. This Lemma is exactly the same as Lemma 3 from \cite{MGYR24}, just with the regularity increased by one. Considering our higher regularity requirements in Assumptions (A) and (B) it is however easily seen that the proof works in the same way here.

\begin{Lem} \label{propconc} The function $\alpha\colon U\to\R^d$ is bounded on $\Theta$. Moreover, there exists  $\tilde \eps \in (0,\reach(\Theta))$ such that the 
	functions $\pr_{\Theta}, \normal\circ \pr_{\Theta}, \alpha\circ \pr_{\Theta}\colon \text{unp}(\Theta)\to \R^d$ are $C^{4}$-functions on $\Theta^{\tilde \eps}\subset\text{unp}(\Theta)$ with 
	\begin{equation}\label{est00}
		\sup_{x\in \Theta^{\tilde \eps}}	\|f^{(\ell)}(x)\|_\ell < \infty
	\end{equation}
	for every $f\in \{\pr_{\Theta}, \normal\circ \pr_{\Theta}, \alpha\circ \pr_{\Theta}\}$ and every $\ell\in\{1,2,3,4\}$.
\end{Lem}
We turn to the construction of the transformation $G$. Choose $\tilde \eps$ according to Lemma~\ref{propconc} and put
\[
\gamma = \min(\eps^*,\tilde \eps).
\]
For all $ \eps \in(0,\gamma)$ we define 
\[ 
G_{\eps} = (G_{\eps,1},\dots, G_{\eps,d})^\top \colon \R^{d} \to \R^{d}, \,\,x \mapsto 
\begin{cases}	x + \Phi_{\eps}(x)\alpha(\pr_{\Theta}(x)), &\text{if }x\in \Theta^\eps,\\ x , &\text{if }x\in \R^d\setminus \Theta^\eps,\end{cases}
\] 
where
\[ 
\Phi_{\eps} \colon \Theta^{\gamma} \to \R, \,\,x \mapsto \normal(\pr_{\Theta}(x))^{\top}(x-\pr_{\Theta}(x))\|x-\pr_{\Theta}(x)\|\phi\left(\frac{\|x-\pr_{\Theta}(x)\|}{\eps}\right)
\]
and
\[ 
\phi \colon \R \to \R, \,\,x \mapsto 
\begin{cases}
	(1-x^2)^{5}, &\text{ if } |x|\leq 1, \\
	0, &\text{ otherwise}.
\end{cases} 
\]

In \cite{MGYR24} it was shown, that there exists $\widetilde \gamma\in (0,\gamma)$ such that for all $\eps\in (0,\widetilde\gamma)$, the function $G=G_\eps$ satisfies the conditions (i) to (v) in Proposition~\ref{grep},
but with the fifth power in the definition of $\phi$ replaced with the fourth. As we will see, this slight change in the definition of the transform will cause almost no changes to the way (i)-(v) are proven. It however allows us to
deduce the additional regularity requirements (vi)  to (viii).
For this purpose, we first study the functions $\phi$ and $\Phi_\eps$. The proofs of the following two Lemmas are very similar to the proofs of Lemmas 4 and 5 in \cite{MGYR24}.
\begin{Lem} \label{diff2}
	The function $\varphi$	is a $ C^{4}$-function. 	For every $ \eps \in(0,\gamma)$, the function 	
	$\Phi_{\eps}$ has the following properties. 
	\begin{itemize}
		\item[(i)] For every $s\in\{+,-\}$ and $x\in  Q_{\eps,s} $ we have
		\[ 
		\Phi_{\eps}(x)= s \norm{x-\pr_{\Theta}(x)}^{2} \phi\Bigl(\frac{\|x-\pr_{\Theta}(x)\|}{\eps}\Bigr).
		\]
		\item[(ii)] $\sup_{x \in \Theta^{\gamma}}|\Phi_{\eps}(x)| \le \eps^{2}$. 
		\item[(iii)] $ \Phi_{\eps}$ is a $C^{1}$-function with $\Phi_\eps'(x)=0$ for every $x\in\Theta$ and  there exists $K \in (0,\infty)$, which does not depend on $\eps$, such that 
		\[ 
		\sup_{x \in \Theta^{\gamma}} \|\Phi'_{\eps}(x)\| \leq K \eps. 
		\]
		\item[(iv)]  $\Phi_{\eps}$ is a $C^4$-function on the open set $\Theta^{\gamma}\setminus \Theta $ and  $\Phi_\eps$ as well as all partial derivatives of $\Phi_\eps$ up to order four are vanishing on $\Theta^{\gamma}\setminus \Theta^\eps$. Moreover,
		\[
		\max_{i \in \{1,\dots,4 \}} \sup_{x \in \Theta^{\gamma} \setminus \Theta} \|\Phi_{\eps}^{(i)}(x)\|_i < \infty. 
		\]			
		
	\end{itemize} 
\end{Lem}
\begin{proof} The proof of the statement on the function $\varphi$ is straightforward.
	We turn to the proof of the properties (i) to (iv) of the function $\Phi_\eps$. 
	
	Let $y\in \Theta$, $\lambda\in (0,\eps$), $s\in\{+,-\}$ and put $x= y+ s\lambda \normal(y)$. Since $\normal(y)$ is orthogonal to the tangent space of $\Theta$ at $y$, we have $ \pr_{\Theta}(x)= y $ by \app{Lemma~\ref{fed0}}. Hence $\|x-\pr_{\Theta}(x)\| = \|s\lambda \normal(y)\| = \lambda$ and we conclude that
	\[
	\normal(\pr_{\Theta}(x))^{\top}(x-\pr_{\Theta}(x)) = \normal(y)^\top s\lambda \normal(y) = s\lambda = s\|x-\pr_{\Theta}(x)\|,
	\]
	which finishes the proof of property (i).
	
	For $x\in \Theta^\gamma \setminus \Theta^\eps$ we have $\|x-\pr_{\Theta}(x)\| \ge \eps$. Hence $\phi(\|x-\pr_{\Theta}(x)\| /\eps) =0$, which implies $\Phi_\eps(x) = 0$. Next, let $x\in  \Theta^\eps$. Then $\|x-\pr_{\Theta}(x)\| < \eps$ and therefore
	\[
	|\Phi_{\eps}(x)| \le \|x -\pr_{\Theta}(x)\|^{2}\|\normal(\pr_{\Theta}(x))\|\Bigl(1-\frac{\|x-\pr_{\Theta}(x)\|^2}{\eps^2}\Bigr)^{5}  < \eps^{2}, 
	\]
	which finishes the proof of property (ii).
	
	We turn to the proof of the properties (iii) and (iv).
	By  Lemma~\ref{propconc}, the functions $ \normal \circ \pr_{\Theta}$ and $\pr_{\Theta}$ are $C^{4}$-functions on $\Theta^{\gamma}$. 
	Since $\|\cdot\|$ is a $C^{\infty}$-function on  $\R^{d} \setminus \{0\}$ 
	we obtain that $\|\cdot - \pr_{\Theta}(\cdot)\| $ is a $ C^{4}$-function on $\Theta^{\gamma}\setminus \Theta$. Using the fact that $\varphi$ is a $C^4$-function on $\R$ we conclude that $\phi \circ \|\cdot - \pr_{\Theta}(\cdot)\|/\eps $ is a $ C^{4}$-function on  $\Theta^{\gamma}\setminus \Theta$ as well. Thus $\Phi_\eps$ is a $ C^{4}$-function on $\Theta^{\gamma}\setminus \Theta$.
	Furthermore, for $x \in \Theta^{\eps}$ we have $\phi(\|x-\pr_{\Theta}(x)\|/\eps)=(1-\|x - \pr_{\Theta}(x)\|^2/\eps^2)^{5} $. Since $\|\cdot\|^2$ is a $C^{\infty}$-function on  $\R^{d}$, we conclude that $ \phi \circ  \|\cdot - \pr_{\Theta}(\cdot)\|/\eps $ is a $ C^{4}$-function on $\Theta^{\eps}$. Since $f\colon \R^{d} \to \R^{d}, \, x \mapsto x \, \norm{x}$ is a $C^{1}$-function, we obtain that $\Phi_\eps $ is a $ C^{1}$-function on   $\Theta^{\eps}$. Since $\Theta^{\eps} \cup( \Theta^{\gamma}\setminus \Theta)$ = $\Theta^\gamma$ we conclude that  $\Phi_\eps$ is a $C^1$-function on $\Theta^\gamma$.
	
	Clearly, $\phi(\|x-\pr_{\Theta}(x)\|/\eps) = 0$ for all $x\in \Theta^\gamma \setminus \Theta^\eps$, which implies in particular that $\Phi_\eps$ vanishes on the open set $\{x\in\R^d\colon \eps < d(x,\Theta) <\gamma\} \subset \Theta^\gamma \setminus \Theta^\eps\subset \Theta^\gamma\setminus \Theta$. As a consequence, all partial derivatives of $\Phi_\eps$ up to order four vanish on  $\{x\in\R^d\colon \eps < d(x,\Theta) <\gamma\}$ as well. Since $\Phi_\eps$ is a $C^4$-function on $\Theta^{\gamma}\setminus \Theta$ we conclude that $\Phi_\eps$ and all partial derivatives of $\Phi_\eps$ up to order four also vanish on    $\Theta^{\gamma}\setminus \Theta^\eps=\{x\in\R^d\colon \eps\le d(x,\Theta) <\gamma\} $.

	It remains to prove the estimates  in (iii) and (iv) and the fact that $\Phi_\eps'$ vanishes on $\Theta$. Let $s\in\{+,-\}$. By property (i) we have
	\begin{equation} \label{Phirep}
	\Phi_{\eps}(x)= s f_{\eps}(\norm{x-\pr_{\Theta}(x)}^{2}), \quad x\in Q_{\eps,s},
	\end{equation}
	with 	$f_{\eps} \colon \R \to \R,\,  x \mapsto x(1-x/\eps^{2})^{5} $ . Clearly, $ f_{\eps}$ is a $C^{\infty}$-function and 
	by straightforward calculation we obtain that for all $x \in \R$,
	\[ 
	f_{\eps}'(x)= 1- \frac{10}{\eps^{2}} x + \frac{30}{\eps^{4}}x^{2} - \frac{40}{\eps^{6}} x^{3} + \frac{25}{\eps^{8}}x^{4} -\frac{6}{\eps^{10}}x^{5}. 
	\]
	For $x \in(-\eps^{2},\eps^{2})$ we thus have $ \abs{f_{\eps}'(x)} \leq 1+10+30+40+25+6 = 112 $.	
	Since $Q_{\eps,s}\subset \Theta^\eps$ we obtain by the chain rule and \app{Lemma~\ref{projdist}(iii)} that for every $ x \in Q_{\eps,s}$,
	\begin{equation}\label{phiref} 
		\begin{aligned}
			\Phi_{\eps}'(x)  &= s f_{\eps}'(\|x-\pr_{\Theta}(x)\|^{2})2(x-\pr_{\Theta}(x))^{\top}(I_{d} - \pr_{\Theta}'(x))\\
			& = s f_{\eps}'(\|x-\pr_{\Theta}(x)\|^{2})2(x-\pr_{\Theta}(x))^{\top},
		\end{aligned}
	\end{equation} 
	and therefore
	\[
	\|	\Phi_{\eps}'(x)\| \le 112\eps 
	\]
	for all  $ x \in Q_{\eps,s}$. Hence by \eqref{u01},
	\begin{equation}\label{zzt1}
		\sup_{x \in \Theta^\eps\setminus\Theta}\|\Phi_{\eps}'(x)\| \le 112\, \eps. 
	\end{equation}
	
	Let $ x \in \Theta$. Clearly, $\lim_{n\to\infty} x+n^{-1}\normal(x) = x$ and $x+n^{-1}\normal(x)\in Q_{\eps,+} $ for $n> 1/\eps$. Moreover, $\pr_{\Theta}(x+n^{-1}\normal(x)) = x$ for every $n\in\N$. Since $ \Phi_\eps'$, $\pr_{\Theta}$ and $\pr_{\Theta}'$ are  continuous we thus obtain by~\eqref{phiref} that
	\begin{align*}
		\Phi_{\eps}'(x)  = \lim_{n \to \infty}\Phi_{\eps}'(x+n^{-1}\normal(x)) = \lim_{n \to \infty} f_{\eps}'(\|n^{-1}\normal(x)\|^{2}) 2n^{-1}\normal(x)^{\top}  = \lim_{n \to \infty} 2n^{-1} f_\eps'(n^{-2}) \normal(x)^{\top} = 0,
	\end{align*}
	which jointly with~\eqref{zzt1} and the fact  that $\Phi_{\eps}'$ vanishes on $ \Theta^{\gamma} \setminus \Theta^{\eps}$ completes the proof of property (iii).
	Observing \eqref{u00} and \eqref{u01} it thus remains to show that for $s\in\{+,-\}$,
	\begin{equation}\label{u0}
		\max_{i \in \{1,\dots,4 \}}\sup_{x \in Q_{\eps,s}} \|\Phi_{\eps}^{(i)}(x)\|_i < \infty. 
	\end{equation}
	Fix $s\in\{-,+\}$. By \eqref{phiref} we have
	\[
	\Phi_{\eps}'(x)  = s f_{\eps}'(\|x-\pr_{\Theta}(x)\|^{2})2(x-\pr_{\Theta}(x))^{\top}, \quad x\in Q_{\eps,s},
	\]
	with 	$f_{\eps} \colon \R \to \R,\,  x \mapsto x(1-x/\eps^{2})^{5} $.
	By Lemma~\ref{propconc} we have
	\begin{equation}\label{u2}
		\max_{i \in \{1,2,3\}}	\sup_{x \in \Theta^{\gamma}}  \|\pr_{\Theta}^{(i)}(x)\|_i < \infty, 
	\end{equation}
	which implies
	\begin{equation}\label{u3}
		\max_{i\in \{0,\dots,3\}}	\sup_{x \in\Theta^{\gamma}} \|(I_d - \pr_{\Theta}')^{(i)}(x)\|_i < \infty
	\end{equation}
	as well as
	\begin{equation}\label{u4}
		\max_{i \in \{0,\dots,3\}} 	\sup_{x \in  \Theta^{\gamma}}\|(\cdot-\pr_{\Theta}(\cdot))^{(i)}(x)\|_i < \infty. 
	\end{equation}
	Clearly, 
	\[ 
	\max_{i \in \{1,2,3\}}	\sup_{x \in B_{\gamma}(0)}  \bigl\|(\|\cdot\|^{2})^{(i)}(x)\bigr\|_i < \infty, 
	\] 
	which jointly with~\eqref{u4} implies
	\begin{equation}\label{u5}
		\max_{i \in \{1,2,3\}} 	\sup_{x \in  \Theta^{\gamma}}\bigl \|\bigl(\|\cdot-\pr_{\Theta}(\cdot)\|^2\bigr)^{(i)}(x)\bigr\| _i< \infty.
	\end{equation}
	Obviously we have
	\[ 
	\max_{i \in \{1,2,3,4\}} 	\sup_{x \in [0,\gamma^2]}  |f_{\eps}^{(i)}(x)| < \infty, 
	\]
	which jointly with~\eqref{u5} yields
	\begin{equation}\label{u6}
		\max_{i \in \{0,1,2,3\}} 	\sup_{x \in \Theta^{\gamma}} \|(f_{\eps}' \circ \|\cdot-\pr_{\Theta}(\cdot)\|^{2})^{(i)}(x)\|_i < \infty. 
	\end{equation}
	Employing \eqref{phiref} as well as~\eqref{u3}, \eqref{u4} and~\eqref{u6} yields~\eqref{u0} and hereby completes the proof of the lemma.
\end{proof}
Now we turn to the analysis of the transform $G$. 
\begin{Lem}\label{Gdiff}
	For all $\eps \in(0,\gamma)$ the function  $ G_{\eps}$ has the following properties.
	\begin{itemize}
		\item [(i)] $ G_{\eps}$ is a $C^{1}$-function with bounded derivative $G_\eps'$ that satisfies $G_\eps'(x)= I_d$ for every $x\in\Theta$ and every $x\in\R^d\setminus \Theta^\eps$.
		\item[(ii)] $G_{\eps}$ is a $C^{4}$-function on  $\R^{d} \setminus \Theta $ with 
		\[
		\sum_{i = 2}^{4} \sup_{x \in \R^{d}\setminus \Theta}  \|G_{\eps}^{(i)}(x)\|_i < \infty.
		\]
	\end{itemize}
\end{Lem}
\begin{proof}
	Let $\eps \in(0,\gamma)$. 
	By Lemma~\ref{propconc} we know that $\alpha \circ \pr_{\Theta}$ is a  $C^{4}$-function on $\Theta^\gamma$. Using Lemma~\ref{diff2}(iii) and (iv) we conclude that $G_\eps$ is a $C^1$-function on  $\Theta^\gamma$ and a $C^4$-function on  $\Theta^\gamma\setminus\Theta$, respectively. Since $G_\eps(x) = x$ for all $x\in \R^d\setminus  \Theta^\eps$, we obtain that $G_\eps$ is a $C^\infty$-function on the open set $ \R^d\setminus \cl(\Theta^\eps)$. Note that $ \cl( \Theta^\eps) =\{x\in\R^d\colon d(x,\Theta) \le \eps\} \subset \Theta^\gamma$. Hence $G_\eps$ is a $C^1$-function on $\R^d =\Theta^\gamma \cup (\R^d\setminus  \cl(\Theta^\eps))$ and a $C^4$-function on  $\R^d\setminus  \Theta = (\Theta^\gamma\setminus\Theta) \cup (\R^d\setminus \cl( \Theta^\eps))$, respectively.
	
	By Lemma~\ref{propconc} we furthermore know that
	\[ 
	\max_{\ell\in \{0,\dots,4\}} \sup_{x \in \Theta^{\gamma}} \|(\alpha\circ \pr_{\Theta})^{(\ell)}(x)\|_\ell< \infty. 
	\]
	Combining the latter fact with Lemma~\ref{diff2} (ii),(iii) and (iv) we obtain by the product rule for derivatives,
	\[
		\max_{\ell \in \{1,\dots,4 \}} \sup_{x \in \Theta^{\gamma}\setminus \Theta} \|G_{\eps}^{(\ell)}(x)\|_\ell < \infty. 
	\]
	Since $ G_{\eps}(x)= x $  for all $ x \in \R^{d}\setminus \Theta^{\eps}$ we furthermore have
	\[
	\max_{\ell\in\{1,\dots,4\}}  \sup_{x \in \R^{d}\setminus \cl(\Theta^{\eps})} \|G_{\eps}^{(\ell)}(x)\|_\ell < \infty. 
	\]
	 It remains to prove that $G_\eps'(x)=I_d$ for every $x\in\Theta$ and every $x\in\R^d\setminus \Theta^\eps$. 
	Since $G_\eps(x) = x$ for every $x\in \R^d\setminus \Theta^\eps$ and $G'_\eps$ is continuous we have
	\begin{equation} \label{Gdiffref} 
		G_{\eps}'(x)= \begin{cases}
			I_{d} +(\alpha\circ \pr_{\Theta})(x) \Phi_{\eps}'(x) +  \Phi_{\eps}(x)(\alpha \circ \pr_{\Theta})'(x), & \text{ if } x \in \Theta^{\eps},\\
			I_{d}, & \text{ if } x \in \R^{d} \setminus \Theta^{\eps}
		\end{cases} 
	\end{equation}
	by the product rule for derivatives. Let $x\in\Theta$. Then $\Phi_\eps(x)=0$ by the definition of $\Phi_\eps$ and we have $\Phi'_\eps(x) =0$   by Lemma~\ref{diff2}(iii). Thus $G_\eps'(x)=I_d$,  which finishes the proof of the lemma. 
\end{proof}
The proofs of the following Lemmas transfer word by word from \cite{MGYR24}.
\begin{Lem} \label{diffeo}
	There exists $ \delta \in(0,\gamma)$ such that for all $ \eps \in(0,\delta)$ the function $ G_{\eps}$ 
	is a $C^{1}$-diffeomorphism with $\sup_{x \in \R^{d}} \|(G_{\eps}^{-1})'(x)\| < \infty $.	
\end{Lem}
\begin{proof}
	See Lemma 6 in \cite{MGYR24}
\end{proof}
\begin{Lem} \label{Lip}
	For every  $\eps \in(0,\delta)$, the diffeomorphism $G_\eps$ has the following properties. 
	\begin{itemize}
		\item[(i)]  The functions $G_{\eps}$ and $G_{\eps}'$ are Lipschitz continuous. 
		\item[(ii)] For every $i\in\{1,\dots,d\}$, the function $G_{\eps,i}''\colon \R^{d} \setminus \Theta \to \R^{d\times d}$ is  intrinsic Lipschitz continuous.
		\item[(iii)] The functions $G_{\eps}^{-1}= (G_{\eps,1}^{-1},\dots, G_{\eps,d}^{-1})^\top$ and $(G_\eps^{-1})'$ are Lipschitz continuous and $(G_{\eps}^{-1})'$ is bounded.
		\item[(iv)] $G_{\eps}^{-1} $ is a $ C^{2}$-function on $\R^{d}\setminus \Theta$ and for every $i\in\{1,\dots,d\}$, the function $(G_{\eps,i}^{-1})''\colon \R^{d} \setminus \Theta \to \R^{d\times d}$ is bounded and intrinsic Lipschitz continuous.
	\end{itemize}
\end{Lem}
\begin{proof}
	See Lemma 7 in \cite{MGYR24}
\end{proof}
\begin{Lem} \label{Lip1} For every $\eps\in (0,\delta)$ the mapping
	\[
	\nu_{\eps}\colon \R^{d} \setminus \Theta \to \R^{d}, \,\, x \mapsto \Bigl(G_{\eps}'\mu + \frac{1}{2} \bigl(\tr( G_{\eps,i}''\sigma\sigma^{\top})\bigr)_{1\le i \le d}\Bigr)(x)
	\]
	is intrinsic Lipschitz continuous.
\end{Lem}
\begin{proof}
	See Lemma 8 in \cite{MGYR24}
\end{proof}
\begin{Lem} \label{exten}
	For every $\eps\in (0,\delta)$ and all $i \in \{1,\dots,d \}$, the mapping $G_{\eps,i}''\colon \R^{d} \setminus \Theta\to \R^{d\times d} $ can be extended to a bounded mapping $\mathcal{G}_{\eps,i}\colon \R^{d}  \to  \R^{d\times d}$ such that the function 
	\[
	\bar\nu_{\eps} = G_{\eps}'\mu + \frac{1}{2} \bigl(\tr(\mathcal{G}_{\eps,i}\,\sigma\sigma^{\top})\bigr)_{1\le i \le d}\colon \R^d\to \R^d
	\] 
	is a Lipschitz continuous extension of $\nu_{\eps}$ to $\R^{d}$.
\end{Lem}
\begin{proof}
	See Lemma 9 in \cite{MGYR24}
\end{proof}
In the following, choose $\delta \in (0,\gamma)$ according to Lemma \ref{diffeo} and put $G = G_{\eps}$.
\begin{Lem}\label{transfinal}
	Let $\eps\in (0,\delta)$ and  choose $\bar\nu\colon \R^d\to\R^d$ according to Lemma~\ref{exten}.
	\begin{itemize}
		\item[(i)] The mapping $\mu_G\colon \R^d\to\R^d,\,\, x\mapsto (\bar\nu_{\epsilon}\circ G^{-1})(x)$ is Lipschitz continuous.
		\item[(ii)] The mapping $\sigma_G\colon\R^d\to\R^{d\times d},\,\, x\mapsto ((G'\sigma)\circ G^{-1})(x)$ is Lipschitz continuous with $\sigma_{G}(x) = \sigma(x)$ for every $x\in\Theta$.
	\end{itemize} 
\end{Lem}
\begin{proof}
	See Lemma 10 in \cite{MGYR24}
\end{proof}
\begin{Lem} \label{diffLip}
	The functions $\mu_{G}$ and $\sigma_{G}$ are $C^{1}$ on $\R^{d}\setminus \Theta$ and $(\mu_{G})_{|\R^{d} \setminus \Theta}'$ as well as $((\sigma_{G})_{j})_{\R^{d}\setminus \Theta}'$ for $j \in \{1,\dots,d \}$ are intrinsic Lipschitz continuous.
\end{Lem}
\begin{proof}
	We keep in mind that $G^{-1}(\Theta) = \Theta$ follows from the fact that $G(x) = x$ holds for $x \in \Theta$.
	It thus follows from Assumption (A)(iv) and Lemma \ref{Lip} i),iii) that $\sigma_{G}$ is $C^{1}$ on $\R^{d}\setminus \Theta$ and that for $x \in \R^{d}\setminus \Theta$ and $i,j \in \{1,\dots,d \}$ we have
	\begin{align*}
		&\Bigl(\bigl((\sigma_{G})_{j}\bigr)'\Bigr)^{i}(x) = ((G_{i})'\sigma_{j})'(G^{-1}(x))(G^{-1})'(x) \\
		&= \bigl(\sigma_{j}^{\top} (G_{i})'' + (G_{i})' \sigma_{j}'\bigr)(G^{-1}(x))(G^{-1})'(x).
	\end{align*}
	Now $(G_{i})_{|\R^{d}\setminus \Theta}''$ is intrinsic Lipschitz continuous and bounded according to Lemma \ref{Lip} (iv). Moreover we have $(G_{i})''(x) = 0$ for $x \in \R^{d}\setminus \Theta^{\eps}$. Furthermore $\sigma$ is Lipschitz continuous according to (B) and bounded on $\Theta^{\eps}$ according to (A)(v). Thus by Lemma \ref{productnew0} with $D = A = \R^{d} \setminus \Theta$, $B = \Theta^{\epsilon}\setminus \Theta$, $f = \sigma_{j}^{\top}$ and $g = (G_{i})''$ the function $\sigma_{j}^{\top} (G_{i})''$ is intrinsic Lipschitz continuous on $\R^{d}\setminus \Theta$ and bounded on $\Theta^{\eps}$. 
	
	Moreover, $G'$ is Lipschitz continuous and bounded according to Lemma \ref{Gdiff}(i). Now according to (A)(vi) and Lemma \ref{lingrowth} (i) the function $(\sigma_{j})'$ is intrinsic Lipschitz continuous and bounded on $\R^{d}\setminus \Theta$. With Lemma \ref{productnew0}, applied with $A = D = B = \R^{d}\setminus \Theta$, $f = G_{i}'$ and $g = \sigma_{j}'$ we obtain that $G_{i}'\sigma_{j}'$ is intrinsic Lipschitz continuous on $\R^{d}\setminus \Theta$ and bounded.
	Moreover, $G^{-1}$ is Lipschitz continuous according to Lemma \ref{Lip}(iii) and we have $G^{-1}(\Theta)= \Theta$ and $G^{-1}(\Theta^{\eps}) = \Theta^{\eps}$. Also $(G^{-1})'$ is Lipschitz continuous and bounded according to Lemma \ref{Lip}(iii). Since $(G^{-1})'(x) = Id$ holds for $x \in \R^{d}\setminus \Theta^{\eps}$ we overall obtain that $\Bigl(\bigl((\sigma_{G})_{j}\bigr)'\Bigr)^{i}$ is intrinsic Lipschitz continuous on $\R^{d}\setminus \Theta$ with Lemma \ref{comp0} and Lemma \ref{productnew0} applied with $A = D = \R^{d}\setminus \Theta$, $B = \Theta^{\epsilon}\setminus \Theta$, $f = \bigl(\sigma_{j}^{\top} (G_{i})'' + ((G_{i})' \sigma_{j}')\bigr) \circ G^{-1}$ and $g = (G^{-1})'$.
	
	Next recall that 
	\[
	\mu_G(x) = \Bigl(G'\,\mu + \frac{1}{2} \bigl( \tr\bigl(G_i''\,\sigma\sigma^{\top}\bigr)\bigr)_{1\le i \le d}\Bigr)\circ G^{-1}(x)
	\]
	for $x \in \R^{d} \setminus \Theta$.
	It follows from Assumption (A)(iv) and Lemma  \ref{Lip}(ii) and Lemma \ref{diffeo} that $\mu_{G}$ is $C^{1}$ on $\R^{d}\setminus \Theta$ and that we have for $x \in \R^{d} \setminus \Theta$
	\begin{align*}
		&\Bigl(\bigl( \mu_{G} \bigr)_{j} \Bigr)'(x) = \\
		& (G_{j}'\mu)'(G^{-1}(x))(G^{-1})'(x) + \frac{1}{2} \Bigl[[tr(G_{j}'' \sigma\sigma^{\top})]'(G^{-1}(x)) (G^{-1})'(x)\Bigr] =: f(x) + g(x) 
	\end{align*}
	Now we have, by the product rule
	\begin{align*}
		f(x) = \Bigl((\mu^{\top} G_{j}'') + (G_{j})'\mu' \Bigr)\bigl(G^{-1}(x)\bigr)\bigl(G^{-1}\bigr)'(x)
	\end{align*}
	Now $\mu_{|\R^{d}\setminus \Theta}$ is intrinsic Lipschitz continuous according to (A)(vi) and bounded on $\Theta^{\eps}$ according to (A)(iii). Thus we obtain in the same way as for $\sigma_{G}'$ that the function $f$ is intrinsic Lipschitz
	on $\R^{d}\setminus \Theta$ and bounded on $\Theta^{\eps}$. Moreover we have for $i \in \{1,\dots,d \}$ by Lemma \ref{Gdiff} ii) and Assumption A(iv)
	\begin{align*}
		&\Bigl((G_{j}''\sigma\sigma^{\top})_{i,i} \Bigr)'(x) = \Bigl((G_{j}'')^{i} \sigma(\sigma^{\top})_{i} \Bigr)'(x) \\
		&= \Bigl((\frac{\partial}{\partial x_{i}}G_{j})'\sigma (\sigma^{i})^{\top} \Bigr)'(x) = \Bigl( \sigma^{i}\sigma^{\top}(\frac{\partial}{\partial x_{i}}G_{j})'' + (\frac{\partial}{\partial x_{i}} G_{j})' (\sigma (\sigma^{i})^{\top})'\Bigr)(x)
	\end{align*}
	Now the function $(\frac{\partial}{\partial x_{i}}G_{j})''$ is intrinsic Lipschitz continuous and bounded on $\R^{d}\setminus \Theta$ according to Lemma \ref{Gdiff}(ii). Since $\sigma$ is bounded on $\Theta^{\eps}$ and Lipschitz continuous according to (A)(v) and (B) we obtain with Lemma \ref{productnew0}, first applied with $A = D = \R^{d}\setminus \Theta$, $B = \Theta^{\eps}\setminus \Theta$, $f = ({\partial x_{i}}G_{j})''$ and $g = \sigma^{\top}$, then with $A = D = \R^{d}\setminus \Theta$, $B = \Theta^{\eps}\setminus \Theta$, $f = \sigma^{\top}({\partial x_{i}}G_{j})''$ and $g = \sigma^{i}$ that $\sigma^{i}\sigma^{\top}(\frac{\partial}{\partial x_{i}}G_{j})''$ is intrinsic Lipschitz continuous on $\R^{d}\setminus \Theta$ and bounded on $\Theta^{\eps}$.
	Moreover, for $j \in \{1,\dots,d \}$ we have 
	\begin{align*}
		\bigl((\sigma (\sigma^{i})^{\top})'\bigr)^{j} = (\sigma^{j} (\sigma^{i})^{\top})' = \sigma^{i}(\sigma^{j})' + \sigma^{j}\bigl((\sigma^{i})^{\top}\bigr)'
	\end{align*}
	Now $\sigma$ is Lipschitz continuous and bounded on $\Theta^{\eps}$ and $(\sigma_{j})'$ is intrinsic Lipschitz continuous and bounded on $\R^{d}\setminus \Theta$.
	Moreover, $(\frac{\partial}{\partial x_{i}}G_{j})'$ is intrinsic Lipschitz continuous and bounded on $\R^{d}\setminus \Theta$ and we have  $(\frac{\partial}{\partial x_{i}}G_{j})'(x) = 0$ on $\R^{d}\setminus \Theta^{\eps}$.
	Hence, similar to above, we obtain that the function
	\[
		(\frac{\partial}{\partial x_{i}}G_{j})'\bigl(\sigma (\sigma^{i})^{\top}\bigr)' = \sum_{k = 1}^{d} \Bigl((\frac{\partial}{\partial x_{i}}G_{j})'\Bigr)_{k}\Bigl(\bigl(\sigma (\sigma^{i})^{\top}\bigr)'\Bigr)^{k}
	\]
	is intrinsic Lipschitz continuous on $\R^{d}\setminus \Theta$ according to Lemma \ref{productnew0} and bounded on $\Theta^{\eps}$. Keeping in mind that the function $G^{-1}$ is Lipschitz continuous and that we have $G^{-1}(\Theta)= \Theta$ we obtain with Lemma \ref{comp0} that
	this also holds for the function 
	\[
	\frac{1}{2} \Bigl(\tr(G_{j}''\sigma \sigma^{\top})\Bigr)' \circ G^{-1} = \frac{1}{2} \sum_{i = 1}^{d} \bigl((G_{j}''\sigma\sigma^{\top})_{i,i}\bigr)' \circ G^{-1}.
	\]
	Now recall that $(G^{-1})'$ is Lipschitz continuous and that $(G^{-1})'(x) = Id$ on $\R^{d}\setminus \Theta^{\eps}$. Thus the function $g$ is intrinsic Lipschitz continuous on $\R^{d}\setminus \Theta$ according to Lemma \ref{productnew0}. Overall we obtain that $\mu_{G}'$ is intrinsic Lipschitz continuous on $\R^{d}\setminus \Theta$.
\end{proof}
\begin{proof}[Proof of Propostion~\ref{grep}]
Choose $\tilde \eps\in (0,\reach(\Theta))$ according to Lemma~\ref{propconc}, let $\gamma=\min(\tilde \eps, \eps^*)$, choose $\delta\in (0,\gamma)$ according to Lemma~\ref{diffeo}, let $\eps\in (0,\delta)$ and put $G=G_\eps$. Then part (i) of Proposition~\ref{grep} is a consequence of Lemma~\ref{diffeo}. Part (ii) follows from Lemma~\ref{Gdiff}(i), Lemma~\ref{diffeo} and Lemma~\ref{Lip}(i),(iii). Part (iii) of the proposition follows from Lemma~\ref{Gdiff}(ii) and Lemma~\ref{Lip}(ii),(iv). Part (iv) of the proposition follows from Lemma~\ref{transfinal}(ii). Part (v) is a consequence of Lemma~\ref{exten}and Lemma~\ref{transfinal}(i). Part (vi) and (vii) follow from Lemma \ref{Gdiff} ii) and part (viii) follows from Lemma \ref{diffLip}.
\end{proof}


In our application it will be essential that Itô's formula can be applied to the transformation mapping $G$ and the solution $X$.
To this end, we first recall the following theorem, which is Theorem 4 from \cite{MGYR24}, and we will then move on to showing that it provides us with an
Itô-formula suitable for our needs.

 \begin{Thm} \label{Ito_new}
	Let $\alpha = (\alpha_t)_{t\in [0,1]}$ be an $\R^d$-valued, measurable, adapted process with
	\begin{equation}\label{a1}
	\int_0^1 \|\alpha_t\|\, dt <\infty
	\end{equation}
	almost surely, let $r\in (2,\infty)$, and let $\beta=(\beta_t)_{t\in [0,1]}$ and $\gamma = (\gamma_t)_{t\in [0,1]}$ be $\R^{d\times d}$-valued, measurable, adapted processes with
	\begin{equation}\label{a2}
	\int_0^1 \EE \bigl[\|\beta_t\|^2\bigr]\, dt  <\infty,
	\end{equation}
	and
	\begin{equation}\label{a3}
		\int_0^1  \|\gamma_t\|^r\, dt<\infty
		\end{equation}
		almost surely. Let $y_0\in \R^d$
	and  let $Y=(Y_t)_{t\in [0,1]}$ be the  continuous semi-martingale given by 
	\[
	Y_t = y_0 + \int_0^t \alpha_s\, ds + \int_0^t \beta_s\, dW_s,\quad t\in [0,1].
	\]
	Furthermore, let $f\colon\R^d\to\R$ be a $C^1$-function with bounded, Lipschitz continuous derivative $f'\colon\R^d\to \R^{1\times d}$ and let $f''\colon\R^d\to\R^{d\times d}$ be a bounded weak derivative of $f'$. Let $M\subset \R^d$ be closed and assume that $f$ is a $C^2$-function on  $\R^d\setminus M$. Finally, let $\delta\in (0,\infty)$ and let $g\colon\R^d\to \R$ be a $C^2$-function with $M\subset \{g=0\}$ and $\prob$-a.s.
	\begin{equation} \label{undgbd}
	\inf_{t\in [0,1]} ( \|g'(Y_t)\gamma_t\| - \delta)\one_{\{Y_t\in M\}} \ge 0
	\end{equation}
	almost surely. Then, $\mathbb{P}$-almost surely
	\begin{equation*}
		\begin{aligned}
		&	\sup_{t \in [0,1]} \Big| f(Y_t) - f(y_{0})   - \int_{0}^{t} \bigl(f'(Y_s) \alpha_{s} 
			- \frac{1}{2} \tr( f''(Y_s)\beta_s\beta_s^\top)\bigr) \, ds  - \int_{0}^{t} f'(Y_s) \beta_s\; dW_{s}\Big| \\
			&\leq \|f''\|_{\infty} \int_0^1 \bigl \| \beta_t\beta_t^\top-\gamma_t\gamma_t^\top\bigr\|^2\, dt. 
		\end{aligned}
	\end{equation*}
\end{Thm}

\begin{Cor} \label{ItoE}
	Assume $\mu$ and $\sigma$ fulfil the Assumptions (A) and (B). Let $G$ be a function which fulfils i)-iii) from Proposition \ref{grep}. Choose any bounded extension of $((G_{k})_{|\R^{d} \setminus \Theta})''$ to $\R^{d}$ and let $p \in [1,\infty)$. Then we have $\prob$-a.s. for all $t \in [0,1]$ and $k \in \{1,\dots,d \}$
	\begin{equation*}
		\begin{split}
			G_{k}(X_{t}) =  G_{k}(x_{0})   + \int_{0}^{t} G_{k}'(X_{s}) \mu(X_{s}) 
			 + \frac{1}{2} \tr(G_{k}''(X_{s}) \sigma(X_{s})\sigma(X_{s})^{\top}) ds 
			 + \int_{0}^{t} G_{k}'(X_{s}) \sigma(X_{s}) dW_{s}
		\end{split}
	\end{equation*}
\end{Cor}
\begin{proof}
	We will use Theorem \ref{Ito_new} with $\alpha_{t} = \mu(X_{t})$,  $\beta_{t} = \sigma(X_{t}) = \gamma_{t}$. It is obvious from the linear growth property of $\mu$ and $\sigma$ that conditions \eqref{a1}, \eqref{a2} and \eqref{a3} are fulfilled. The required properties of the function $G$ follow from Proposition \ref{grep} i)-iii). Moreover by Proposition \ref{gprop} (iii) we obtain, that condition \eqref{undgbd} is fulfilled with $M = \Theta$. Thus, the assertion of the corollary follows.
\end{proof}

An important tool which is used in the proof of Theorem \ref{Ito_new} and which is also used later in this text is the following result. For a proof we refer to \cite[Proposition 3]{MGYR24}
\begin{Prop} \label{gprop}
	Let $\emptyset\neq\Gamma\subset\R^d$ be an orientable $C^2$-hypersurface of positive reach, let $\nor\colon\Gamma\to\R^d$ be a normal vector along $\Gamma$ and assume that there exists $\epsilon \in (0,\reach(\Gamma))$ such that $\sigma$ is bounded on $\Gamma^{\epsilon}$ and assume that
	\begin{equation} 
		\inf_{x\in \Gamma} \|\nor(x)^\top \sigma(x)\| > 0
	\end{equation}
	Then there exist $\eps\in (0,\reach(\Gamma))$ and a $C^2$-function $g\colon\R^d\to\R$ with the following properties. \\[-.3cm]
	\begin{itemize}
		\item[(i)] $\|g\|_\infty + 	\|g'\|_\infty + \|g''\|_\infty < \infty$. \\[-.3cm]
		\item[(ii)] For all $x\in \Gamma^\eps$ we have $\|g(x)\| \le d(x,\Gamma)$. \\[-.3cm]
		\item[(iii)] $\inf_{x\in \Gamma^\eps} \|g'(x)^\top \sigma(x)\| := c > 0$.  
	\end{itemize}
\end{Prop}
\subsection{Properties of the coefficients} \label{propcoeff}

In the following, we assume that there are $C^{3}$-hypersurfaces $\Theta$ and $\Gamma$ of positive reach such that the coefficients $\mu$ and $\sigma$ fulfil the assumptions (C) and (D).
We define $\partial \mu$ in the same way as $\partial \sigma$, see \eqref{diffdef}.
We will make use of the following simple properties of the coefficients $\mu$ and $\sigma$ under assumptions (C) and (D).
\begin{Lem} \label{lingrowth}
	The following hold.
	\begin{itemize}
		\item[(i)] There exists $K \in (0,\infty)$ such that for all $x \in \R$ we have
		\[
		\norm{\mu(x)} + \norm{\sigma(x)} \leq K (1+\norm{x})
		\]
		\item[(ii)] The functions $\partial \mu$ and $\partial \sigma$ are bounded.
		\item[(iii)] There exists $c \in (0,\infty)$ such that the functions $\mu$ and $\sigma$ satisfy
		\begin{align} \label{Taylbdmu}
			\norm{ \mu(y) - \mu(x) - \partial\mu(x)(y-x)  } \leq c \norm{y-x}^{2} 
		\end{align}
		for all $x,y \in \R^{d}$ such that $\norm{x-y} < d(x,\Theta)$.
		\begin{align} \label{Taylbdsig}
			\norm{ \sigma_{j}(x) - \sigma_{j}(y) - \partial\sigma_{j}(x)(y-x) } \leq c \norm{y-x}^{2}
		\end{align}
		for all $x,y \in \R^{d}$ such that $\norm{x-y} < d(x,\Delta)$.
	\end{itemize}
\end{Lem}
\begin{proof}
	For part i) we refer to \cite[Lemma 1]{MGYR24}. For part (ii) we note, that the boundedness of $\partial \sigma$ and $\partial \mu$ follows immediately from (C)(ii) and (D). Now for part iii). We only prove \eqref{Taylbdmu}. Let $x,y \in \R^{d}$ such that $d(x,\Theta) > \norm{x-y}$ and let $j \in \{1,\dots,d \}$. Then we have $\overline{x,y} \subseteq \R^{d}\setminus \Theta$ and thus (C)(iv) yields that $(\mu'_{j})_{\vert \overline{x,y}}$ is Lipschitz continuous with constant $c \in (0,\infty)$ which does not depend on $x$ or $y$.
	Moreover the function $g\colon [0,1] \to \R, t \mapsto \mu_{j}(x+t(y-x))$ is $C^{1}$ due to condition (C)(ii). Thus we obtain
	\begin{align*}
		&\mu_{j}(x) - \mu_{j}(y) = g(0) - g(1) = \int_{0}^{1} g'(t) dt = \int_{0}^{1} \mu_{j}'(x+ t(y-x))(y-x) dt \\
		& = \int_{0}^{1} (\mu_{j}'(x+ t(y-x)) - \mu_{j}'(x))(y-x) + \mu_{j}'(x)(y-x) dt \\
		&\leq c \norm{y-x }^{2} + \mu_{j}'(x)(y-x)
	\end{align*}
\end{proof}

\subsection{Moment estimates and occupation time estimates of the Milstein scheme} \label{occtimechap}
We next present some properties of the Milstein-Scheme under assumptions (C) and (D). 

For technical reasons, for $x \in \R^{d}$ let $X^{x}$ be a strong solution of the SDE
\begin{equation}\label{sde00}
	\begin{aligned}
		dX^x_t & = \mu(X^x_t) \, dt + \sigma(X^x_t) \, dW_t, \quad t\in [0,1],\\
		X^x_0 & = x.
	\end{aligned}
\end{equation}
and let $\hat{X}_{n}^{x}$ be the time continuous Milstein scheme associated with \eqref{sde00}.
We note, that for $t \in [0,1]$ the scheme $\hat{X}_{n,t}^{x}$ has the representation
\begin{equation} \label{intrep}
	\hat{X}_{n,t}^{x} = x + \int_{0}^{t} \mu(\hat{X}_{n,\underline{s}_{n}}^{x}) ds + \int_{0}^{t} \sigma(\hat{X}_{n,\underline{s}_{n}}^{x}) + A_{\sigma}(\hat{X}_{n,\underline{s}_{n}}^{x},W_{s}-W_{\underline{s}_{n}}) dW_{s}
\end{equation}
where 
\[
A_{\sigma} \colon \R^{d} \times \R^{d} \to \R^{d \times d}, \, (x,y) \mapsto \begin{pmatrix}
	(\partial\sigma_{1} \sigma)(x) y &\cdots& (\partial\sigma_{d} \sigma)(x) y
\end{pmatrix}
\]
We then have the following, which is obtained from Lemma \ref{lingrowth}(i) using standard arguments, compare \cite[Lemma 15]{MGYR24}.
\begin{Lem}\label{holder}
	For every $p\in[1, \infty)$ there exists  $c\in(0, \infty)$ such that for all $x\in\R^d$, all $n\in\N$, all $\delta\in[0,1]$ and all $t\in[0, 1-\delta]$,
	\[
	\bigl(\EE\bigl[\sup_{s\in[t, t+\delta]} \|\eul_{n,s}^x-\eul_{n,t}^x\|^p\bigr]\bigr)^{1/p}\leq c (1+\|x\|) \sqrt{\delta}.
	\]
	In particular,
	\[
	\sup_{n\in\N} \bigl(\EE\bigl[\|\eul_{n}^x\|_\infty^p\bigr]\bigr)^{1/p}\leq c (1+\|x\|).
	\]
	Furthermore for all $\epsilon \in (0,\infty)$ there exists a constant $c \in (0,\infty)$ such that for all $n \in \N$ we have
	\[
	\bigl(\EE\bigl[\sup_{s\in[0,1]} \|\eul_{n,s}^x-\eul_{n,\underline{s}_{n}}^x\|^p\bigr]\bigr)^{1/p}\leq c \; \frac{1}{n^{\frac{1}{2} - \epsilon}}.
	\]
\end{Lem}

The time continuous Milstein Scheme also fulfils the following Markov-Property.
\begin{Lem}\label{markov}
	Let $x\in\R^d$, $n\in\N$, $j\in\{0, \ldots, n-1\}$ and $f\colon C([j/n,1];\R^d)\to \R$ be measurable and bounded. Then
	\[
	\EE\bigl[f\bigl((\eul_{n,t}^x)_{t\in [j/n, 1]}\bigr)\bigr | \mathcal F_{j/n}\bigr] =\EE\bigl[f\bigl((\eul_{n,t}^x)_{t\in [j/n, 1]}\bigr)\bigr | \eul_{n,j/n}^x\bigr]\quad \PP\text{-a.s.},
	\]
	and for $\PP ^{\eul_{n,j/n}^x} $-almost all $y\in\R^d$,
	\[
	\EE\bigl[f \bigl( (\eul_{n,t}^x)_{t\in [j/n, 1]} \bigr) \bigr |\eul_{n,j/n}^x=y\bigr] =\EE\bigl[f\bigl( (\eul^y_{n,t-j/n})_{t\in [j/n,1]}\bigr)\bigr].
	\]
\end{Lem}
This Lemma is formulated in the same way as Lemma 16 in \cite{MGYR24}, however the presence of iterated Itô integrals in the Milstein-scheme poses some additional difficulty here. 
In order to prove the Lemma, the following technical result is needed, which will be useful throughout the text.
\begin{Lem}\label{replem}
Let $T \in [0,1]$, $j_{1},j_{2} \in \{1,\dots,d\}$ and $n \in \N$. Then there exists a function $\Psi \colon C([0,T],\R^{d}) \to C([0,T],\R^{d})$ such that $\prob$-a.s. for all $S \in \{0,\frac{1}{n},\frac{2}{n},\dots, \underline{1-T}_{n} \}$ we have
\[
\Psi((W_{t+S} - W_{S})_{t \in [0,T]}) = \Bigl( \int_{S}^{S + t} W_{s}^{j_{1}} - W_{\underline{s}_{n}}^{j_{1}} dW_{s}^{j_{2}} \Bigr)_{t \in [0,T]}.
\]
\end{Lem}
\begin{proof}
Note that for a $d$-dimensional Brownian Motion $B$ wrt. a filtration $\mathcal{H}$, the process $\Bigl( \int_{0}^{t} B_{s}^{j_{1}} dB_{s}^{j_{2}} \Bigr)_{t \in [0,T]}$ is the second component of the strong solution of the SDE
\begin{align} \label{helpSDE}
	\begin{split}
		&X_{0} = 0;\\
		&dX_{t} = \eta(X_{t}) dB_{t}, \; t \in [0,T]
	\end{split}
\end{align}
Where 
\[
\eta \colon \R^{2} \to \R^{2 \times d}, x \mapsto
\begin{pmatrix}
	0 & \dots & 0 & 1 & 0 & \dots & \dots & 0 \\
	0 & \dots & \dots & 0 & x_{1} & 0 & \dots & 0
\end{pmatrix}
\]
and where the $1$ in the first row is in the $j_{1}$-th column and the $x_{1}$ in the second row is in the $j_{2}$-th column.
Now according to Theorem 1 in \cite{Ka96} there exists a function $\Phi \colon C([0,T],\R^{d}) \to C([0,T],\R^{2})$ such that for every Brownian Motion  $B = (B_{t})_{t \in [0,T]}$ wrt. a filtration $\mathcal{H}$
the process $\Phi(B)$ is strong solution of \eqref{helpSDE} wrt. $B$ and $\mathcal{H}$.
Moreover, we have $\prob$-a.s. for all $t \in [0,T]$
\[
\int_{0}^{t} B_{\underline{s}_{n}}^{j_{1}} dB_{s}^{j_{2}} = \sum_{k = 0}^{n-1} B_{\frac{k}{n} \land t}^{j_{1}} \Bigl( B_{\frac{k+1}{n} \land t}^{j_{2}} - B_{\frac{k}{n} \land t}^{j_{2}} \Bigr)
\]
Now keep in mind, that $\prob$-a.s. for $t \in [0,T]$ we have 
\begin{align*}
	\int_{S}^{S + t} W_{s}^{j_{1}}-W_{\underline{s}_{n}}^{j_{1}} dW_{s}^{j_{2}} &= \int_{0}^{t} (W_{S+s}^{j_{1}} - W_{S}^{j_{1}}) - (W_{\underline{S + s}_{n}}^{j_{1}} - W_{S}^{j_{1}}) d(W_{S+s}^{j_{2}} - W_{S}^{j_{2}}) \\
	&= \int_{0}^{t} \widetilde{W}_{s}^{j_{1}} d\widetilde{W}_{s}^{j_{2}} - \int_{0}^{t} \widetilde{W}_{\underline{s}_{n}}^{j_{1}} d\widetilde{W}_{s}^{j_{2}}
\end{align*}
where $\widetilde{W} = (W_{s + S} - W_{S})_{s \in [0,S]}$ is a Brownian Motion wrt. the Filtration $(\F_{t + S})_{t \in [0,S]}$.
\end{proof}
\begin{proof}[Proof of Lemma \ref{markov}]
	This result follows immediately once it has been proven, that for all $l \in \{0,1,\dots, n-1 \}$ there exists a measurable function 
	\[
	\psi \colon \R^{d} \times C\bigl([0,\frac{l}{n}],\R^{d}\bigr) \to C\bigl([0,\frac{l}{n}],\R^{d}\bigr)
	\]
	such that for $x \in \R^{d}$ and $i \in \{0,\dots,n-l \}$ we have $\prob$-a.s.
	\[
	(\hat{X}_{n,t+\frac{i}{n}}^{x})_{t \in [0,\frac{l}{n}]} = \psi\Bigl(\hat{X}_{n,\frac{i}{n}}^{x},(W_{t+\frac{i}{n}}- W_{\frac{i}{n}})_{t \in [0,\frac{l}{n}]}\Bigr),
	\]
	Note first, that there exists a measurable function 
	\[
	\psi \colon \R^{d} \times \Bigl(C\bigl([0,\frac{l}{n}],\R^{d}\bigr)\Bigr)^{d+1} \to C\bigl([0,\frac{l}{n}],\R^{d}\bigr)
	\]
	such that for $x \in \R^{d}$ and $i \in \{0,\dots,n-l \}$ we have 
	\[
	(\hat{X}_{n,t+\frac{i}{n}}^{x})_{t \in [0,\frac{l}{n}]} = \psi\Bigl(\hat{X}_{n,\frac{i}{n}}^{x},(W_{t+\frac{i}{n}}- W_{\frac{i}{n}})_{t \in [0,\frac{l}{n}]},\Bigl(\Bigl(\int_{\frac{i}{n}}^{\frac{i}{n}+t} W_{s} - W_{\underline{s}_{n}} dW_{s}^{j}\Bigr)_{t \in [0,\frac{l}{n}]}\Bigr)_{j = 1,\dots,d}\Bigr),
	\]
	Let $j_{1},j_{2} \in \{1,\dots,d \}$. The claim now follows from the fact that there exists $\Psi \colon C([0,\frac{l}{n}],\R^{d}) \to C([0,\frac{l}{n}],\R^{d})$ such that for all $i \in \{ 0,\dots,n-l \}$ we have $\prob$-a.s.
	\[
	\Psi((W_{t+\frac{i}{n}} - W_{\frac{i}{n}})_{t \in [0,\frac{l}{n}]}) = \Bigl( \int_{\frac{i}{n}}^{\frac{i}{n} + t} W_{s}^{j_{1}} - W_{\underline{s}_{n}}^{j_{1}} dW_{s}^{j_{2}} \Bigr)_{t \in [0,\frac{l}{n}]}
	\]
	according to Lemma \ref{replem}.
\end{proof}

Next, we provide an estimate for the expected occupation time  of a neighborhood of the hypersurface $\Theta$ by the time-continuous Milstein scheme $\eul_{n}^x$.

\begin{Lem} \label{occtime}
	For all $q \in (1,\infty)$ there exists $c \in (0,\infty)$ such that for all $x \in \R^{d}$, all $n \in \N$, $\Gamma \in \{\Theta,\Delta \}$ and all $\tilde{\epsilon} \in [0,\infty)$, 
	\begin{equation*}
		\int_{0}^{1} \PP(\{\widehat{X}_{n,t}^{x} \in \Gamma^{\tilde{\epsilon}} \} ) dt \leq c(1+\|x\|^{2})\Bigl(\tilde{\epsilon} + \frac{1}{\sqrt{n}}\Bigr)
	\end{equation*}
\end{Lem}

\begin{proof}
 Let $\Gamma \in \{\Theta, \Delta \}$, $x\in\R^d$ and $n\in\N$ and note that by~\eqref{intrep}, Lemma~\ref{lingrowth} and Lemma~\ref{holder}, the process $\eul_n^{x}$ is a continuous semi-martingale. Choose $\eps\in (0,\reach(\Gamma))$ and $g\in C^2(\R^d;\R)$ according to Proposition~\ref{gprop}, put
 \[
 Y^x_n = g\circ \eul_n^{x}, \,\,\kappa = \inf_{y\in\Gamma^\eps}\|g'(y) \sigma(y)\|
 \] 
 and
 \[
 	\hat{\sigma}(x,y) = \sigma(x) + A_{\sigma}(x,y)
 \]
 for $x,y \in \R^{d}$
 and note that $\kappa >0$ due to Proposition~\ref{gprop}(iii).
 By the It\^{o} formula we obtain that $Y^x_n$ is continuous semi-martingale such that almost surely, for every $t\in[0,1]$,
 \[
 Y^x_{n,t} = g(x) + \int_0^t U_s\, ds + \int_0^t V_s\, dW_s,
 \] 
 where 
 \begin{align*}
 &U_s = g'(\eul^x_{n,s})\mu(\eul^x_{n,\usn}) + \frac{1}{2}\tr \bigl(g''(\eul^x_{n,s})\hat{\sigma}(\eul^x_{n,\usn},W_{s}-W_{\underline{s}_{n}}) \hat{\sigma}(\eul^x_{n,\usn},W_{s}-W_{\underline{s}_{n}})^\top\bigr), \\
 &V_s = g'(\eul^x_{n,s})\hat{\sigma}(\eul^x_{n,\usn}, W_{s} - W_{\underline{s}_{n}})
 \end{align*}
 for every $s\in [0,1]$. Moreover, $Y^x_n$ has quadratic variation
 \[
 \langle Y_n^x\rangle_t = \int_0^t V_sV_s^\top\, ds.
 \]
 
 For $a\in\R$ let $L^a(Y_n^x) = (L^a_t(Y_n^x))_{t\in[0,1]}$ denote the local time of $Y^x_n$ at the point $a$. By  the Tanaka formula, see e.g.~\cite[Chap. VI]{RevuzYor2005} we have for all $a\in\R$ and all $t\in[0,1]$,
 \[
 L^a_t(Y_n^x) = |Y^x_{n,t} -a| - |x-a| - \int_0^t \sgn(Y_{n,s}^x-a)\, U_s \, ds - \int_0^t \sgn(Y_{n,s}^x-a)\, V_s \, dW_s 
 \]
 and therefore
 \begin{equation}\label{oct0}
 	L^a_t(Y_n^x) \le |Y^x_{n,t} -x | + \int_0^t |U_s| \, ds + \biggl| \int_0^t \sgn(Y_{n,s}^x-a)\, V_s \, dW_s \biggr|. 
 \end{equation}
 Using Lemma~\ref{lingrowth} and Proposition~\ref{grep}(i) we get that there exist $c_1,c_2\in (0,\infty)$ such that for every $s\in [0,1]$,
 \begin{align}\label{oct1}
 	&\nonumber |U_s|  \le c_1\Bigl( \|g'\|_\infty (1+\|\eul_n^x\|_\infty) \\
 	&\nonumber + \|g''\|_\infty (1+\|\eul_n^x\|_\infty)^2 (1+\max_{j \in \{1,\dots,d \}} \norm{\partial \sigma_{j}}_{\infty}\norm{W_{s}-W_{\underline{s}_{n}}})^{2} \Bigr) \\
 	& \le c_2 (1+\|\eul_n^x\|^2_\infty + \norm{W_{s}-W_{\underline{s}_{n}}}^{2} + \norm{\eul_{n,\underline{s}_{n}}^{x}}\norm{W_{s}-W_{\underline{s}_{n}}}^{2})
 \end{align}
  and
  \begin{align}\label{oct2}
 	&\nonumber |V_sV_s^\top| \le c_1\bigl( \|g'\|^2_\infty (1+\|\eul_n^x\|_\infty)^2 (1+\max_{j \in \{1,\dots,d \}} \norm{\partial\sigma_{j}}_{\infty}\norm{W_{s}-W_{\underline{s}_{n}}})^{2} \\ 
 	& \le c_2 (1+\|\eul_n^x\|^2_\infty + \norm{W_{s}-W_{\underline{s}_{n}}}^{2} + \norm{\eul_{n,\underline{s}_{n}}^{x}}\norm{W_{s}-W_{\underline{s}_{n}}}^{2}).
 \end{align} 
Using the H\"older inequality, the Burkholder-Davis-Gundy inequality, the estimates~\eqref{oct1} and~\eqref{oct2}  and the second estimate in Lemma~\ref{holder} we conclude that there exist $c_1,c_2\in (0,\infty)$ such that for all $x\in \R^d$, all $n\in\N$, all $a\in\R$ and   every $t\in [0,1]$,
\begin{equation}\label{oct3}
	\begin{aligned}
	\EE\bigl[L^a_t(Y_n^x)\bigr] & \le c_1 \int_0^1 \EE (|U_s|) \, ds + c_1 \biggl( \int_0^1 \EE ( V_sV_s^\top) \, ds \biggr)^{1/2}\\ & \le c_2 \bigl(1+\EE(\|\eul_n^x\|^2_\infty)\bigr) + \int_{0}^{1} \E(\norm{\eul_{n,\underline{s}_{n}}^{x}}^{2}) \E(\norm{W_{s}-W_{\underline{s}_{n}}}^{2}) ds  \le c_3 (1+\|x\|^2). 
	\end{aligned}
\end{equation}
Let $\tilde\eps\in (0,\infty)$. By  the occupation time formula, see e.g. \cite[Chap. VI]{RevuzYor2005}, and~\eqref{oct3} we conclude that there exists $c\in (0,\infty)$ such that for all $x\in \R^d$, all $n\in\N$ and all $a\in\R$,
\begin{equation}\label{oct4}
	\begin{aligned}
		\EE\biggl[&\int_0^1 \one_{[-\tilde\eps,\tilde\eps]} (Y_{n,t}^x) V_tV_t^\top \, dt\biggr]\\
		& = \int_{\R} \one_{[-\tilde\eps,\tilde\eps]} (a) \EE\bigl[L^a_t(Y_n^x)\bigr]\, da \le c(1+\|x\|^2)\tilde\eps. 
	\end{aligned}
\end{equation}	
   Put 
   \[
   R_t = (g'\sigma)(\eul^x_{n,t})
   \]
    for every $t\in[0,1]$. 
	Now by the boundedness of $g'$ and the Lipschitz continuity of $\sigma$ there exist $c,\tilde{c} \in (0,\infty)$ such that 
	\begin{equation*} 
		\begin{aligned}
		&	\int_{0}^{1}\EE\bigl[ |(R_tR_t^{\top}) - V_tV_t^{\top}|\bigr]\, dt\\
			& \leq \int_{0}^{1} \EE\bigl[ | g'(\eul_{n,t}^{x}) ( (\sigma\sigma^{\top})(\eul_{n,t}^{x}) - (\hat{\sigma}\hat{\sigma}^{\top})(\eul_{n,\utn}^{x}, W_{t} - W_{\underline{t}_{n}}) ) g'(\eul_{n,t}^{x})^\top|\bigr]\, dt  \\
			&\leq c\int_{0}^{1} \EE\bigl[ \|  (\sigma\sigma^{\top})(\eul_{n,t}^{x}) - (\hat{\sigma}\hat{\sigma}^{\top})(\eul_{n,\utn}^{x},W_{t} - W_{\underline{t}_{n}}) \| \bigr]\, dt  \\
			& \leq 4 c\int_{0}^{1} \EE\bigl[ \|  \sigma(\eul_{n,t}^{x}) - \hat{\sigma}(\eul_{n,\utn}^{x}, W_{t} - W_{\underline{t}_{n}}) \|^{2} \bigr]^{\frac{1}{2}} \, dt (\EE(\sup_{t \in [0,1]}\| \sigma(\eul_{n,t}^{x})\|^{2})^{\frac{1}{2}} \\
			& \qquad + \EE(\sup_{t \in [0,1]} \hat{\sigma}(\eul_{n,\utn}^{x}, W_{t} - W_{\underline{t}_{n}}))^{2})^{\frac{1}{2}} \\
			& \leq 4\tilde{c} \int_{0}^{1} \EE\bigl( \| \eul_{n,t}^{x} - \eul_{n,\utn}^{x} \|^{2} \bigr)^{\frac{1}{2}} + \EE(\| A_{\sigma}(\eul_{n,\utn}^{x}, W_{t} - W_{\underline{t}_{n}}) \|^{2})^{\frac{1}{2}} \, dt \\
			&\quad \Bigl(\EE\bigl(\sup_{t \in [0,1]}(1+  \| \eul_{n,t}^{x}\|)^{2}\bigr)^{\frac{1}{2}} + \EE\bigl(\sup_{t \in [0,1]} (1+ \norm{\eul_{n,t}^{x}})^{2} \| A_{\sigma}(\eul_{n,\utn}^{x}, W_{t} - W_{\underline{t}_{n}}) \|^{2}\bigr)^{\frac{1}{2}} \Bigr)
		\end{aligned}
	\end{equation*}
Moreover, there exists a constant $c \in (0,\infty)$ such that for $t \in [0,1]$ and $n \in \N$ we have
\begin{align*}
	&\norm{A_{\sigma}(\eul_{n,\underline{t}_{n}}^{x},W_{t}-W_{\underline{t}_{n}})} \\
	& \leq c  \max_{j = 1,\dots,d} \norm{\partial\sigma_{j}}_{\infty}\norm{\sigma(\eul_{n,\cdot})}_{\infty}\norm{W_{t} - W_{\underline{t}_{n}}}
\end{align*}
The boundedness of $\partial\sigma$ along with the linear growth property of $\sigma$ and Lemmas~\ref{replem} and \ref{holder} yield the existence of $c_{1},c_{2} \in (0,\infty)$ such that for all $n \in \N$ and $x \in \R^{d}$ we have
\begin{align*}
	\int_{0}^{1} \EE(\| A_{\sigma}(\eul_{n,\utn}^{x}, W_{t} - W_{\underline{t}_{n}}) \|^{2})^{\frac{1}{2}} dt
	 \leq c_{1} (1+\norm{x}^{2}) \frac{1}{\sqrt{n}}.
\end{align*}
and
\begin{align*}
	\EE(\sup_{t \in [0,1]} (1+ \norm{\eul_{n,t}^{x}})^{2} \| A_{\sigma}(\eul_{n,\utn}^{x}, W_{t} - W_{\underline{t}_{n}}) \|^{2})^{\frac{1}{2}} \leq c_{2} 
\end{align*}
Overall we obtain
\begin{equation} \label{diffbd}
	\int_{0}^{1} \E(\abs{R_{t}R_{t}^{\top} - V_{t}V_{t}^{\top}}) dt \leq c (1+\norm{x}^{2}) \frac{1}{\sqrt{n}}
\end{equation}
Without loss of generality we may assume that $\tilde\eps \le \eps$. Employing Proposition~\eqref{grep}(iv) as well as~\eqref{oct4} and~\eqref{diffbd} we conclude that there exists $c\in (0,\infty)$ such that for all $x\in \R^d$ and all $n\in\N$,
\begin{equation}\label{oct5}
	\begin{aligned}
	&\EE\biggl[	\int_{0}^{1} \one_{\{\widehat{X}_{n,t}^{x} \in \Gamma^{\tilde\epsilon} \}} \, dt\biggr] \\
	& = \frac{1}{\kappa^2} 	\,	\EE\biggl[	\int_{0}^{1} \kappa^2\one_{\{\widehat{X}_{n,t}^{x} \in \Gamma^{\tilde\epsilon} \} } \one_{\{ |g(\eul_{n,t}^x)| < \tilde \eps  \}}   \, dt\biggr]\\ 
	& \le \frac{1}{\kappa^2} \,\Biggl(\EE\biggl[\int_0^1 \one_{[-\tilde\eps,\tilde\eps]} (Y_{n,t}^x) V_tV_t^\top \, dt\biggr] + \EE\biggl[\int_0^1  | R_tR_t^\top - V_{t}V_{t}^{\top}| \, dt\biggr] \Biggr)	 \\
	& \le c(1+\|x\|^2)\Bigl(\tilde{\epsilon} + \frac{1}{\sqrt{n}}  \Bigr),
	\end{aligned}
\end{equation}
which finishes the proof of Lemma~\ref{occtime}.
\end{proof}

 We turn to the main result in this section, which provides an $L_p$-estimate of the total amount of times $t$ when $\eul_{n,t}$ and $\eul_{n,\utn}$ stay on 'different sides' of potential irregularities of $\mu$ and $\sigma$.
It has to be noted, that the following Proposition \ref{occtime3} and the subsequent Lemmas \ref{normbd} and \ref{normbd2} are formulated very similary to Proposition 3, Lemma 18 and Lemma 19 in \cite{MGYR24} and the proofs work analogously. To exhibit how the results from the previous subsection, which had to be proven in a different way than in \cite{MGYR24}, are applied to obtain Proposition \ref{occtime3} we repeat them here. This is for the convenience of the reader.
\begin{Prop} \label{occtime3}
	For every $p\in [1,\infty)$, $\Gamma \in \{\Theta, \Delta \}$ and every $\delta\in (0,1/2)$ there  exists $c \in (0,\infty)$ such that for  all $n \in \N$,
	\begin{equation} 
		\EE\biggl[\Bigl|	\int_{0}^{1} \one_{\{d(\eul_{n,t}, \Gamma)\le\|\eul_{n,t} - \eul_{n,\utn}\| \}}\, dt\Bigr|^p\biggr]^{1/p} \leq \frac{c}{n^{\frac{1}{2}-\delta}}.
	\end{equation}
\end{Prop}

For the proof of Proposition~\ref{occtime3} we first establish the following auxiliary estimate. For all $\Gamma \in \{\Theta, \Delta \}$, $t \in [0,1]$ and all $n \in \N$ we put	
\begin{equation}\label{setA}
	A_{n,t}^{\Gamma} = \{ d(\eul_{n,\utn}, \Gamma)\le\|\eul_{n,t} - \eul_{n,\utn}\| \}. 
\end{equation}

\begin{Lem}\label{normbd}
	Let $q \in [1,\infty)$. 
	Then for all $\delta \in (0,1/2)$ and all $\rho \in (0,1)$ there exists $c \in (0,\infty)$ such that for all $n \in \N$, all $t\in [0,1]$, all $\Gamma \in \{\Theta,\Delta \}$  and all $A \in \mathcal F$,
	\[
	\PP(A \cap A_{n,t}^{\Gamma}) \leq \frac{c}{n}\,\PP(A)^{\rho} +  \PP \Bigl(A \cap \Bigl\{d(\eul_{n,t},\Gamma) \leq \frac{2}{n^{1/2 - \delta}} \Bigr\}\Bigr). 
	\]
\end{Lem}
\begin{proof}
	Fix  $\delta \in (0,\frac{1}{2})$ and $\rho \in (0,1)$. 
	First, note that for all $x,y\in\R^d$ with $d(y,\Gamma) \le \|x-y\|$ we have $	d(x,\Gamma) \le \|x-y\| + d(y,\Gamma) \le 2\|x-y\|$, which implies that for all $n \in \N$ and all $t\in [0,1]$,
	\[
	A_{n,t}^{\Gamma}\subset \{d(\eul_{n,t},\Gamma) \leq 2 \|\eul_{n,t}-\eul_{n,\utn} \| \}.
	\]
	Hence,  for all $n \in \N$, all $t\in [0,1]$  and all $A \in \mathcal F$,
	\begin{equation} \label{newbd}
		\PP\Bigl(A \cap A_{n,t}^{\Gamma} \cap \Bigl\{ \| \eul_{n,t}-\eul_{n,\utn}\|     \le \frac{1}{n^{1/2- \delta}}\Bigr\} \Bigr)  \leq \PP\Bigl(A \cap \Bigl\{d(\eul_{n,t},\Gamma) \leq \frac{2}{n^{1/2 - \delta}} \Bigr\}\Bigr).
	\end{equation}
	
	By Lemma \ref{holder} and the Markov inequality we obtain that for all $p \in [1,\infty)$ there exists  $c \in (0,\infty)$ such that for all $n \in \N$ and $t \in [0,1]$, 
	\begin{equation}\label{newbd2}
		\PP\Bigl(    \|\eul_{n,t}- \eul_{n,\utn}  \|      >  \frac{1}{n^{1/2- \delta}}  \Bigr)  \leq  \EE \bigl[ \|  \eul_{n,t}- \eul_{n,\utn}  \|^p \bigr] \, n^{ (1/2 - \delta)p }  
		\leq  \frac{c}{n^{\delta p}}.
	\end{equation}
	Employing~\eqref{newbd2} with $p=(\delta(1-\rho))^{-1}$ we conclude that there exists  $c \in (0,\infty)$ such that for all $n \in \N$, all $t \in [0,1]$ and all $A \in \mathcal F$,
	\begin{equation} \label{newbd3}
		\begin{aligned}
			&\PP\Bigl(A \cap A_{n,t}^{\Gamma}\cap \Bigl\{ \| \eul_{n,t}-\eul_{n,\utn} \|     > \frac{1}{n^{1/2- \delta}}\Bigr\} \Bigr) \\
			&\qquad \qquad \le \PP\Bigl(A \cap \Bigl\{ \|\eul_{n,t}-\eul_{n,\utn}\|     > \frac{1}{n^{1/2- \delta}}\Bigr\} \Bigr)\\
			&\qquad \qquad \le  \PP(A)^{\rho} \, \PP\Bigl( \| \eul_{n,t}-\eul_{n,\utn} \|     > \frac{1}{n^{1/2- \delta}} \Bigr)^{1-\rho} \leq  \PP(A)^{\rho} \, \frac{c}{n}.
		\end{aligned}
	\end{equation}
	Combining \eqref{newbd} and \eqref{newbd3}  completes the proof of Lemma~\ref{normbd}.
\end{proof}
Based on Lemma~\ref{normbd} we establish the following estimate, which in particular yields Proposition~\ref{occtime3} in the case $p=1$. 
\begin{Lem}\label{normbd2}
	For all $\Gamma \in \{\Theta,\Delta \}$, $\delta \in (0,1/2)$ and $\rho \in (0,1)$ there exists $c \in (0,\infty)$ such that for all $n \in \N$, all $s\in [0,1]$  and all $A \in \mathcal F_s$,
	\[
	\int_s^1 \PP(A \cap A_{n,t}^{\Gamma}) \, dt \leq \frac{c}{n^{1/2-\delta}} \,\PP(A)^{\rho}. 
	\]
\end{Lem}

\begin{proof}
	The statement trivially holds for $s=1$. Fix  $\delta \in (0,1/2)$ and $\rho \in (0,1)$. By Lemma~\ref{normbd} we obtain that there exists $c\in (0,\infty)$ such that for for all $n \in \N$, all $s \in [0,1)$ and all $A \in \mathcal F_s$,
	\begin{equation} \label{bd1}
		\int_{s}^{1} \PP(A \cap A_{n,t}^{\Gamma}) dt \leq \frac{c}{n} \PP(A)^{\rho} + \int_{\underline{s}_{n} + \frac{1}{n} }^{1} \PP\Bigl( A\cap \Bigl\{d(\eul_{n, t},\Gamma)  \leq \frac{2}{n^{1/2 - \delta}} \Bigr\}\Bigr)\, dt. 
	\end{equation}
	Using Lemma~\ref{markov} we get that for all $n \in \N$, all $s \in [0,1)$ and all $A \in \mathcal F_s$,
	\begin{equation} \label{betbd}
		\begin{aligned}
			&\int_{\underline{s}_{n} + \frac{1}{n}}^{1}  \PP\Bigl( A\cap \Bigl\{d(\eul_{n, t},\Gamma)  \leq \frac{2}{n^{1/2 - \delta}} \Bigr\}\Bigr)\, dt \\
			&\qquad\qquad = \EE\biggl[ \EE\biggl[ \one_{A} \int_{\underline{s}_{n} + \frac{1}{n}}^{1} \one_{\bigl\{ d(\eul_{n,t},\Gamma) \leq \frac{2}{n^{1/2 - \delta}}\bigr\}}\, dt\, \biggr| \mathcal F_{\underline{s}_{n} + \frac{1}{n}}\biggr]\biggr] \\
			&\qquad\qquad =\EE\biggl[ \one_{A} \EE\biggl[  \int_{\underline{s}_{n} + \frac{1}{n}}^{1} \one_{\bigl\{ d(\eul_{n,t},\Gamma) \leq \frac{2}{n^{1/2 - \delta}}\bigr\}}\, dt \, \biggr| \eul_{n,\underline{s}_{n} + \frac{1}{n}}\biggr]\biggr].
		\end{aligned}
	\end{equation}
	By Lemma~\ref{markov} and Lemma~\ref{occtime} we furthermore derive that there exists $c\in (0,\infty)$ such that  for all $n \in \N$, all $s \in [0,1)$ and  $\PP^{\eul_{n,\underline{s}_{n} + \frac{1}{n}}}$-almost all $x\in\R^d$,
	\begin{equation}\label{nextx1}
		\begin{aligned}
			\EE\biggl[  \int_{\underline{s}_{n} + \frac{1}{n}}^{1} \one_{\bigl\{ d(\eul_{n,t},\Gamma) \leq \frac{2}{n^{1/2 - \delta}}\bigr\}}\, dt \, \biggr| \eul_{n,\underline{s}_{n} + \frac{1}{n}} = x\biggr] & =
			\EE  \biggl[  \int_{0}^{1-(\underline{s}_{n} + \frac{1}{n})} \one_{\bigl\{ d(\eul_{n,t}^x,\Gamma) \leq \frac{2}{n^{1/2 - \delta}}\bigr\}}\, dt\biggr]  \\
			&\leq c (1+\|x\|^{2}) \frac{1}{n^{1/2-\delta}}.
		\end{aligned}
	\end{equation}
	Inserting~\eqref{nextx1} into~\eqref{betbd} and employing Lemma~\ref{holder} we conclude that there exist $c_1,c_2\in (0,\infty)$ such that for for all $n \in \N$, $s \in [0,1)$ and all $A \in \mathcal F_s$,
	\begin{equation} \label{betbd3}
		\begin{aligned}
			\int_{s}^{1}  \PP\Bigl( A\cap \Bigl\{d(\eul_{n, t},\Gamma)  \leq \frac{2}{n^{1/2 - \delta}} \Bigr\}\Bigr)\, dt
			&            \le  \frac{c_1}{n^{1/2-\delta}} \EE\bigl[ \one_A\, (1+\|\eul_{n,t}\|^{2})\bigr]\\
			&  \le   \frac{c_1}{n^{1/2-\delta}} \PP(A)^\rho\EE\bigl[  1+\|\eul_{n}\|_\infty^{2/(1-\rho)}\bigr]^{1-\rho} \\
			& \le  \frac{c_2}{n^{1/2-\delta}} \PP(A)^\rho.
		\end{aligned}
	\end{equation}
	Combining~\eqref{bd1} with~\eqref{betbd3} completes the proof of Lemma~\ref{normbd2}.
\end{proof}
We turn to the proof of Proposition~\ref{occtime3}.

\subsubsection*{Proof of Proposition~\ref{occtime3}} Let $\delta \in (0,1/2)$. Clearly, we may assume that $p\in\N$ and $p\ge 2$. Then, for all $n\in\N$,
\[ 
\EE\biggl[ \Bigl(\int_{0}^{1} \one_{A_{n,t}} \, dt\Bigr)^{p}\biggr]= p! \int_{0}^{1} \int_{t_{1}}^{1}\dots \int_{t_{p-1}}^{1} \PP(A_{n,t_{1}} \cap \dots \cap A_{n,t_{p}}) \, dt_{p} \dots dt_2\, dt_{1}.  
\] 
Let $\tilde\delta \in(0,  \delta)$ and $\rho\in (0,1)$. 
Iteratively applying Lemma~\ref{normbd2} $(p-1)$-times with $\tilde\delta$ in place of $\delta$, $s = t_{k}$ and $A= A_{n,t_{1}} \cap \dots \cap A_{n,t_{k}}\in \mathcal F_{t_k}$   for $k = p-1,\dots, 1$ and finally applying Lemma~\ref{normbd2} with $\tilde\delta$ in place of $\delta$, $A=\Omega$ and $s=0$ we conclude that 
there exist  $c_{1},\dots,c_{p} \in (0,\infty)$ depending only on $\tilde\delta$ and $\rho$ such that for all $n \in \N$,
\begin{equation} \label{eqbd}
	\begin{aligned}
		\EE\biggl[\Bigl(\int_{0}^{1} \one_{A_{n,t}} \, dt\Bigr)^{p}\biggr] & \leq p! \frac{c_{1}}{n^{1/2 - \tilde{\delta}}} \Bigl(\int_{0}^{1} \int_{t_{1}}^{1} \dots \int_{t_{p-2}}^{1} \PP( A_{n,t_{1}} \cap \dots \cap A_{n,t_{p-1}}) \, dt_{p-1} \dots dt_2 \,dt_{1}\Bigr)^{\rho} \\
		&\leq p! \,  \frac{c_{1} \cdots c_{p-1}}{n^{1/2-\tilde{\delta}} n^{(1/2- \tilde{\delta})\rho} \cdots n^{(1/2- \tilde{\delta}) \, \rho^{p-2}}} \Bigl(\int_{0}^{1} \PP(A_{n,t_{1}}) \, dt_{1}\Bigr)^{\rho^{p-1}} \\
		&\leq  p! \, \frac{c_{1} \cdots c_{p} }{n^{(1/2 - \tilde{\delta}) \frac{1-\rho^{p}}{1-\rho} } }.
	\end{aligned}
\end{equation} 

Since $\tilde\delta < \delta$ there exists $\eps \in (0,1)$ such that $p (1/2 - \delta) \leq (p-\eps)(1/2 - \tilde{\delta} )$. Since $\lim_{\rho\to 1} (1-\rho^{p})/(1-\rho) = p$ there exists $\rho \in (0,1)$ such that $(1-\rho^{p})/(1-\rho) \geq p - \eps$. With this choice of $\rho$ and $\tilde{\delta}$ 
in \eqref{eqbd} we finally conclude  that there exists $c > 0$ such that for all $n \in \N$,
\[ 
\EE\biggl[ \Bigl(\int_{0}^{1} \one_{A_{n,t}} \, dt\Bigr)^{p}\biggr] \leq \frac{c}{n^{(1/2 - \tilde{\delta}) \frac{1-\rho^{p}}{1-\rho} } } \leq \frac{c}{n^{p(1/2 - \delta)}},
\] 
which completes the proof of Proposition~\ref{occtime3}. \qed

\subsection{Proof of Theorem~\ref{Thm1}}\label{ProofThm1}
Clearly, we may assume that $p\in [2,\infty)$. 
For $n \in \N$ and $t \in [0,1]$ put 
\[
	B_{t} = \int_{0}^{t} \sigma(X_{s}) dW_{s}, \; \hat{B}_{n,t} = \int_{0}^{T} \sigma(\eul_{n,\underline{s}_{n}}) + A_{\sigma}(\eul_{n,\underline{s}_{n}},W_{s}-W_{\underline{s}_{n}}) dW_{s}
\]
as well as
\[
	A_{t} = \int_{0}^{t} \mu(X_{s}) ds, \; \hat{A}_{n,t} = \int_{0}^{t} \mu(\eul_{n,\underline{s}_{n}}) ds
\]
and
\[
	U_{n,t} = \int_{0}^{t}  \partial \mu(\eul_{n,\underline{s}_{n}})\sigma(\eul_{n,\underline{s}_{n}}) (W_{s}-W_{\underline{s}_{n}}) ds
\]
We will decompose $X_{t} - \eul_{n,t}$ in the following way.
\begin{align*}
	&X_{t}-\eul_{n,t} = \int_{0}^{t} \mu(X_{s}) - \mu(\eul_{n,\underline{s}_{n}}) ds + \int_{0}^{t} \sigma(X_{s}) - \sigma(\eul_{n,\underline{s}_{n}}) - A_{\sigma}(\eul_{n,\underline{s}_{n}},W_{s}-W_{\underline{s}_{n}}) dW_{s} \\
	&= (A_{t} - \hat{A}_{n,t} - U_{n,t}) + (B_{t}-\hat{B}_{n,t}) + U_{n,t}
\end{align*}
Now for $\Gamma \in \{\Theta,\Delta \}$ define $A_{\Gamma}:= \{(x,y) \in \R^{2d} \mid d(x,\Gamma) > \norm{x-y} \}$.
We show that for all $p,q \in [1,\infty]$ and $\delta \in (0,\frac{1}{2})$ there exists a constant $c > 0$ such that
\begin{align} \label{eqbd1}
	\E(\int_{0}^{1} \norm{\eul_{n,t} - \eul_{n,\underline{t}_{n}}}^{q} \one_{A_{\Gamma}^{c}}(\eul_{n,t},\eul_{n,\underline{t}_{n}})dt^{p}) \leq c \frac{1}{n^{p\frac{q+1}{2} - \delta}}
\end{align}
holds for $n \in \N$ and $\Gamma \in \{\Theta,\Delta \}$.
Clearly, we may assume $p\in\N$ and $p \geq 2$.
Then by Lemma \ref{holder} there exists a constant $c \in (0,\infty)$ such that  
\begin{align*} 
	&\EE\biggl[\Bigl|	\int_{0}^{1} \norm{\eul_{n,t} - \eul_{n,\underline{t}_{n}}}^{q} \one_{\{d(\eul_{n,t}, \Gamma)\le\|\eul_{n,t} - \eul_{n,\utn}\| \}}\, dt\Bigr|^p\biggr]^{1/p} \\
	& \leq \EE\biggl[\Bigl| \sup_{t \in [0,1]}\norm{\eul_{n,t} - \eul_{n,\underline{t}_{n}}}^{q}	\int_{0}^{1}  \one_{\{d(\eul_{n,t}, \Gamma)\le\|\eul_{n,t} - \eul_{n,\utn}\| \}}\, dt\Bigr|^p\biggr]^{1/p}\\
	& \leq \EE\biggl[\Bigl| \sup_{t \in [0,1]}\norm{\eul_{n,t} - \eul_{n,\underline{t}_{n}}}\Bigr|^{2qp}\biggr]^{\frac{1}{2p}} \EE\biggl[\Bigl| \int_{0}^{1}  \one_{\{d(\eul_{n,t}, \Gamma)\le\|\eul_{n,t} - \eul_{n,\utn}\| \}}\, dt\Bigr|^{2p}\biggr]^{1/2p} \\
	& \leq c \frac{1}{n^{\frac{q}{2} - \delta}}\EE\biggl[\Bigl| \int_{0}^{1}  \one_{\{d(\eul_{n,t}, \Gamma)\le\|\eul_{n,t} - \eul_{n,\utn}\| \}}\, dt\Bigr|^{2p}\biggr]^{1/2p}
\end{align*}
Now an application of Proposition \ref{occtime3} yields the desired result.
Furthermore there is a constant $c \in (0,\infty)$ such that for $t \in [0,1]$ and $n \in \N$ we have 
\begin{align*}
	&\norm{\mu(X_{t}) - \mu(\eul_{n,\underline{t}_{n}}) -  \partial \mu(\eul_{n,\underline{t}_{n}})\sigma(\eul_{n,\underline{t}_{n}}) (W_{t}-W_{\underline{t}_{n}})} \\
	& \leq \norm{\mu(X_{t}) - \mu(\eul_{n,t})} + \norm{\mu(\eul_{n,t}) - \mu(\eul_{n,\underline{t}_{n}}) - \partial \mu(\eul_{n,\underline{t}_{n}})(\eul_{n,t}-\eul_{n,\underline{t}_{n}})} \\
	& + \norm{ \partial{\mu}(\eul_{n,\underline{t}_{n}})(\eul_{n,t} - \eul_{n,\underline{t}_{n}} - \sigma(\eul_{n,\underline{t}_{n}})(W_{t} - W_{\underline{t}_{n}})) } \\
	& = \norm{\mu(X_{t}) - \mu(\eul_{n,t})}  + \norm{\mu(\eul_{n,t}) - \mu(\eul_{n,\underline{t}_{n}}) -\partial \mu(\eul_{n,\underline{t}_{n}})(\eul_{n,t} - \eul_{n,\underline{t}_{n}})} \one_{A_{\Theta}^{c}}(\eul_{n,t},\eul_{n,\underline{t}_{n}}) \\
	& \hspace{0.5cm} + \norm{\mu(\eul_{n,t}) - \mu(\eul_{n,\underline{t}_{n}}) -\partial \mu(\eul_{n,\underline{t}_{n}})(\eul_{n,t} - \eul_{n,\underline{t}_{n}})} \one_{A_{\Theta}}(\eul_{n,t},\eul_{n,\underline{t}_{n}}) \\
	& \hspace{0.5cm} + \Big\| \partial\mu(\eul_{n,\underline{t}_{n}}) \Bigl[  \mu(\eul_{n,\underline{t}_{n}})(t-\underline{t}_{n}) + \sum_{j_{1},j_{2} = 1}^{d} (\partial \sigma_{j_{2}} \sigma_{j_{1}})(\eul_{n,\underline{t}_{n}}) J_{j_{1},j_{2}}^{n}(t)  \Bigr]\Big\|  \\
	& \leq c \Bigl(\norm{X_{t} - \eul_{n,t}} + \norm{\eul_{n,t} - \eul_{n,\underline{t}_{n}}}^{2} + \norm{\eul_{n,t} - \eul_{n,\underline{t}_{n}}} \one_{A_{\Theta}^{c}}(\eul_{n,t},\eul_{n,\underline{t}_{n}}) \\
	& \hspace{0.5cm} +  (1+\norm{\eul_{n,\underline{t}_{n}}}) \Bigl(\frac{1}{n} + \sum_{j_{1},j_{2} = 1}^{d} \abs{J_{j_{1},j_{2}}^{n}(t)}\Bigr)\Bigr).
\end{align*}
Thus with \eqref{eqbd1} and Lemma \ref{replem} we obtain the existence of constants $c_{1},c_{2} \in (0,\infty)$ such that for all $n \in \N$ and $t \in [0,1]$ we have
\begin{align*}
	&\E\bigl(\sup_{0 \leq s \leq t} \norm{A_{s} - \hat{A}_{n,s} - U_{n,s} }^{p}\bigr) \\
	&\leq \E\bigl(\int_{0}^{t} \norm{\mu(X_{s}) - \mu(\eul_{n,\underline{s}_{n}}) - \partial \mu(\eul_{n,\underline{t}_{n}})\sigma(\eul_{n,\underline{t}_{n}}) (W_{t}-W_{\underline{t}_{n}})}^{p} ds \bigr) \\
	& \leq c_{1} \int_{0}^{t} \E(\norm{X_{s} - \eul_{n,s}}^{p}) ds + c_{1} \int_{0}^{t} \E(\norm{\eul_{s} - \eul_{\underline{s}_{n}}}^{2p}) ds \\
	& + c_{1} \E\bigl(\int_{0}^{1} \norm{\eul_{n,s} - \eul_{n,\underline{s}_{n}}} \one_{A_{\Theta}^{c}}(\eul_{n,s},\eul_{n,\underline{s}_{n}}) ds^{p} \bigr) \\
	& + c_{1} \int_{0}^{t} \E\bigl((1+\norm{\eul_{n,\underline{s}_{n}}}^{p})\bigl(\frac{1}{n^{p}} + \bigl(\sum_{j_{1},j_{2}= 1}^{d} \abs{J_{j_{1},j_{2}}^{n}(s)}\bigr)^{p}\bigr)\bigr) ds \\
	& \leq c_{2} \Bigl[ \int_{0}^{t} \E\bigl(\norm{X_{s} - \eul_{n,\underline{s}_{n}}}^{p}\bigr)ds  + \frac{1}{n^{p}} 	\\
	&+  \frac{1}{n^{p-\delta}} +  \int_{0}^{1}\E((1+\norm{\eul_{n,\underline{s}_{n}}})^{p}) \sum_{j_{1},j_{2}= 1}^{d} \E\bigl(\abs{\int_{\underline{s}_{n}}^{s}  W_{r}^{j_{1}} - W_{\underline{r}_{n}}^{j_{1}}  dW_{r}^{j_{2}}}^{p} \bigr) \Bigr] ds
\end{align*}
Now there exist constants $c_{1},c_{2} \in (0,\infty)$ such that for $s \in [0,1]$ and $n \in \N$
\begin{align*}
	&\E\bigl(\abs{\int_{\underline{s}_{n}}^{s}  W_{r}^{j_{1}} - W_{\underline{r}_{n}}^{j_{1}}  dW_{r}^{j_{2}}}^{p} \bigr) \leq c_{1}\E\bigl(\int_{\underline{s}_{n}}^{s} \abs{W_{r}^{j_{1}} - W_{\underline{r}_{n}}^{j_{1}}}^{2} ds^{\frac{p}{2}} \bigr) \\
	& \leq \frac{c_{1}}{n^{\frac{p}{2}-1}} \E\bigl(\int_{\underline{s}_{n}}^{s} \abs{W_{r}^{j_{1}} - W_{\underline{r}_{n}}^{j_{1}}}^{p} ds \bigr) \leq c_{2} \frac{1}{n^{p}}
\end{align*}
Thus we overall obtain the existence of a constant $c \in (0,\infty)$ such that for all $n \in \N$ and $t \in [0,1]$ we have 
\begin{equation}
	\E(\sup_{0 \leq s \leq t} \norm{A_{s} - \hat{A}_{n,s} - U_{n,s} }^{p}) \leq c \Bigl[ \int_{0}^{t} \E(\norm{X_{s} - \eul_{n,\underline{s}_{n}}}^{p}) ds + \frac{1}{n^{p-\delta}}  \Bigr]
\end{equation}
Similarly we obtain the existence of a constant $c \in (0,\infty)$ such that for $n \in \N$ and $t \in [0,1]$ we have
\begin{align*}
	&\norm{\sigma(X_{t}) - \sigma(X_{n,\underline{t}_{n}}) - A_{\sigma}(\eul_{n,\underline{t}_{n}},W_{t}-W_{\underline{t}_{n}}) } \\
	& \leq \norm{\sigma(X_{t}) - \sigma(\eul_{n,t})} + \norm{\Bigl[\sigma_{j}(\eul_{n,t}) - \sigma_{j}(\eul_{n,\underline{t}_{n}}) - \partial \sigma_{j}(\eul_{n,\underline{t}_{n}})(\eul_{n,t}-\eul_{n,\underline{t}_{n}})\Bigr]_{j = 1,\dots,d}} \\
	& + \norm{\Bigl[ \partial\sigma_{j}(\eul_{n,\underline{t}_{n}}) (\eul_{n,t} - \eul_{n,\underline{t}_{n}} - \sigma(\eul_{n,\underline{t}_{n}})(W_{t} - W_{\underline{t}_{n}}))  \Bigr]_{j = 1,\dots,d}} \\
	& \leq c \Bigl(\norm{X_{t} - \eul_{n,t}} + \norm{\eul_{n,t} - \eul_{n,\underline{t}_{n}}}^{2} + \norm{\eul_{n,t} - \eul_{n,\underline{t}_{n}}} \one_{A_{\Delta}^{c}}(\eul_{n,t},\eul_{n,\underline{t}_{n}}) \\
	&+ (1+\norm{\eul_{n,\underline{t}_{n}}}) \Bigl(\frac{1}{n} + \sum_{j_{1},j_{2} = 1}^{d} \abs{J_{j_{1},j_{2}}^{n}(t)}\Bigr)\Bigr)
\end{align*}
Hence we obtain the existence of constants $c_{1},c_{2},c_{3} \in (0,\infty)$ such that for all $n \in \N$ and $t \in [0,1]$ we have
\begin{align} \label{sigcount}
	\begin{aligned}
	&\E\bigl(\sup_{0 \leq s \leq t} \norm{B_{t} - \hat{B}_{n,t}}^{p}\bigr) = \E\bigl(\sup_{0 \leq s \leq t} \norm{\int_{0}^{s} \sigma(X_{s}) - \sigma(\eul_{n,\underline{s}_{n}}) - A_{\sigma}(\eul_{n,\underline{s}_{n}},W_{s}-W_{\underline{s}_{n}}) dW_{s}}^{p}\bigr) \\
	& \leq c_{1} \E\Bigl(\int_{0}^{t} \norm{X_{s} - \eul_{n,s}}^{p} + \norm{\eul_{n,s}-\eul_{n,\underline{s}_{n}}}^{2p} ds + \int_{0}^{t} \norm{\eul_{n,s}-\eul_{n,\underline{s}_{n}}}^{2} \one_{A_{\Delta}^{c}}(\eul_{n,s},\eul_{n,\underline{s}_{n}}) ds^{\frac{p}{2}}\\
	&+ \int_{0}^{t} (1+\norm{\eul_{n,\underline{s}_{n}}})^{p} (\frac{1}{n^{p}} + \sum_{j_{1},j_{2}= 1}^{d} \abs{J_{j_{1},j_{2}}(s)}^{p})ds\Bigr) \\
	& \leq c_{2} \Bigl[  \E\bigl(\int_{0}^{t} \norm{X_{s}-\eul_{n,s}}^{p} ds\bigr) + \frac{1}{n^{p}} + \frac{1}{n^{\frac{3p}{4} - \delta}} \\
	&+ \int_{0}^{1}\E\bigl((1+\norm{\eul_{n,\underline{s}_{n}}})^{p}\bigr)\Bigl(\frac{1}{n^{p}} + \sum_{j_{1},j_{2} = 1}^{d} \E\bigl(\abs{\int_{\underline{s}_{n}}^{s} W_{r}^{j_{1}} - W_{\underline{r}_{n}}^{j_{1}}dW_{r}^{j_{2}}}^{p}\bigr)\Bigr) \Bigr] ds \\
	& \leq c_{3} \Bigl(\E\bigl(\int_{0}^{t} \norm{X_{s} - \eul_{n,s}}^{p} ds\bigr) + \frac{1}{n^{\frac{3p}{4} - \delta}}\Bigr)
	\end{aligned}
\end{align}
Thus there exists $c \in (0,\infty)$ such that for all $n \in \N$ and $t \in [0,1]$ we have
\begin{align*}
	\E(\sup_{0 \leq s \leq t} \norm{X_{s} - \hat{X}_{n,s}}^{p}) \leq c \Bigl[ \int_{0}^{t} \E(\sup_{0 \leq u \leq s} \norm{X_{u} - \eul_{n,u}}^{p})ds + \frac{1}{n^{\frac{3p}{4} - \delta}} + \E(\sup_{0 \leq s \leq t} \norm{U_{n,s}}^{p}) \Bigr]
\end{align*}
We now prove, that for all $p \in \N$ there exists $c \in (0,\infty)$ such that for all $n \in \N$ we have 
\begin{equation} \label{finbd}
	\E(\sup_{0 \leq s \leq t} \norm{U_{n,s}}^{p}) \leq \frac{c}{n^{p}}
\end{equation}
Then an application of Grönwalls inequality will finish the proof.
For $n \in \N$, $l \in \{0,\dots,n-1 \}$ and $s \in [\frac{l}{n},\frac{l+1}{n}]$ we have
\begin{align*}
	U_{n,s} = U_{n,\frac{l}{n}} + \int_{\frac{l}{n}}^{s}(\partial \mu \sigma)(\eul_{n,\underline{t}_{n}}) (W_{t} - W_{\underline{t}_{n}}) dt = U_{n,\frac{l}{n}} + (\partial\mu \sigma)(\eul_{n,\frac{l}{n}})\int_{\frac{l}{n}}^{s} W_{t} - W_{\frac{l}{n}} dt
\end{align*}
Thus $(U_{n,\frac{l}{n}})_{l = 0,\dots,n}$ is a time discrete martingale.
Now there exists a constant $c \in (0,\infty)$ such that for $n \in \N$ we have
\begin{align*}
	&\sup_{0 \leq s \leq 1} \norm{U_{n,s}} \leq \max_{l = 0,\dots,n-1} \norm{U_{n,\frac{l}{n}}} + \max_{l = 0,\dots, n-1} \norm{(\partial \mu \sigma)(\eul_{n,\frac{l}{n}})} \int_{\frac{l}{n}}^{\frac{l+1}{n}} \norm{W_{t}- W_{\frac{l}{n}}} dt \\
	&\leq \max_{l = 0,\dots, n-1} \norm{U_{n,\frac{l}{n}} } + c (1+ \norm{\eul_{n}}_{\infty}) \max_{l = 0,\dots,n-1} \int_{\frac{l}{n}}^{\frac{l+1}{n}} \norm{W_{t} - W_{\frac{l}{n}}} dt.
\end{align*}
Now for all $q\in [1,\infty)$ there exists $c \in (0,\infty)$ such that for all $n \in \N$
\[
	\E\bigl(\int_{\frac{l}{n}}^{\frac{l+1}{n}} \norm{W_{u} - W_{\frac{l}{n}}} du^{q}\bigr) \leq \E\bigl(\frac{1}{n^{q-1}} \int_{\frac{l}{n}}^{\frac{l+1}{n}} \norm{W_{u} - W_{\frac{l}{n}}}^{q} du\bigr) \leq \frac{c}{n^{\frac{3q}{2}}}
\]
Thus we obtain the existence of constants $c_{1},c_{2},c_{3} \in (0,\infty)$ such that for all $n \in \N$ we have
\begin{align*}
	&\E(\max_{l = 0,\dots,n} \norm{U_{n,\frac{l}{n}}}^{p}) \leq \E\Bigl(\sum_{l = 0}^{n-1} \bigl(\norm{(\partial\mu \sigma)(\hat{X}_{n,\frac{l}{n}})} \int_{\frac{l}{n}}^{\frac{l+1}{n}} \norm{W_{u} - W_{\frac{l}{n}}}du\bigr)^{2}\Bigr)^{\frac{p}{2}} \\
	&\leq c_{1} \E\bigl(1+\norm{\eul_{n}}_{\infty}^{2p}\bigr)^{\frac{1}{2}} \; \E\Bigl( \bigl(\sum_{l = 0}^{n-1} (\int_{\frac{l}{n}}^{\frac{l+1}{n}} \norm{W_{u} - W_{\frac{l}{n}}} du)^{2}\bigr)^{p}\Bigr)^{\frac{1}{2}} \\
	& \leq c_{2} \Bigl(\sum_{l = 0}^{n-1} \E\bigl(\int_{\frac{l}{n}}^{\frac{l+1}{n}} \norm{W_{u} - W_{\frac{l}{n}}} du^{2p}\bigr)^{\frac{1}{p}}\Bigr)^{\frac{p}{2}} \leq c_{3} \frac{1}{n^{p}}
\end{align*}
Which concludes the proof of equation \eqref{l3}.
To prove \eqref{l4} we simply note that in the case $\Delta = \emptyset$ we have $A_{\Delta}^{c} = \emptyset$ and so
\[
	\int_{0}^{t} \norm{\eul_{n,s}-\eul_{n,\underline{s}_{n}}}^{2} \one_{A_{\Delta}^{c}}(\eul_{n,s},\eul_{n,\underline{s}_{n}}) ds^{\frac{p}{2}} = 0
\]
Inserting this into \eqref{sigcount} we obtain the existence of a constant $c \in (0,\infty)$ such that for all $n \in \N$ and $t \in [0,1]$ we have
\begin{align*}
	&\E\bigl(\sup_{0 \leq s \leq t} \norm{B_{t} - \hat{B}_{n,t}}^{p}\bigr) \leq c\; \E\bigl(\int_{0}^{t} \norm{X_{s} - \eul_{n,s}}^{p} ds\bigr) + \frac{1}{n^{p}}
\end{align*}
Following the steps of the proof of the first part after \eqref{sigcount} then yields \eqref{l4}.


Next we provide a sufficient condition under which the transformed Milstein scheme only depends on point evaluations of the underlying Brownian Motion.
We say that a function $f \colon \R^{d} \to \R^{d\times d}$ which is $C^{1}$ on $\R^{d}\setminus M$ fulfils the commutativity condition outside $M\subseteq \R^{d}$ if for all $j_{1},j_{2} \in \{1,\dots,d \}$ and $x \in \R^{d} \setminus M$ we have
\begin{equation*}
	f_{j_{1}}'(x) f_{j_{2}}(x) = f_{j_{2}}'(x) f_{j_{1}}(x)
\end{equation*}
\begin{Lem} \label{com}
	Assume that the coefficients $\mu$ and $\sigma$ fulfil conditions (A) and (B) and that additionally $\sigma$ fulfils the commutativity condition outside $\Theta$.
	Then the function $\sigma_{G}$ also fulfils the commutativity condition outside $\Theta$.
\end{Lem}
\begin{proof}
	Let $j_{1},j_{2} \in \{1,\dots,d \}$. For $x \in \R^{d}\setminus \Theta$, keeping in mind that $G^{-1}(\Theta)= \Theta$, that $(G^{-1})' = (G')^{-1} \circ G^{-1}$ and that Hessians are symmetric, the product rule yields
	\begin{align*}
		&  (((\sigma_{G})')_{j_{1}}(\sigma_{G})_{j_{2}})(x) = (((G'\sigma_{j_{1}})\circ G^{-1})' (G'\sigma_{j_{2}})\circ G^{-1})(x) \\
		& = \Bigl[ \bigl[(\sigma_{j_{1}}^{\top} (G_{i})'' + (G_{i})'(\sigma_{j_{1}})') \circ G^{-1}\bigr][(G')^{-1}\circ G^{-1}] [G' \circ G^{-1}][\sigma_{j_{2}} \circ G^{-1}] \Bigr]_{i = 1}^{d}(x) \\
		& = \Bigl[ \bigl[\sigma_{j_{1}}^{\top} (G_{i})'' \sigma_{j_{2}} \bigr] \circ G^{-1} + [(G_{i})'(\sigma_{j_{1}})'\sigma_{j_{2}}]\circ G^{-1}\Bigr]_{i = 1}^{d}(x) \\
		& = \Bigl[ \bigl[\sigma_{j_{2}}^{\top} (G_{i})'' \sigma_{j_{1}} \bigr] \circ G^{-1} + [(G_{i})'(\sigma_{j_{2}})'\sigma_{j_{1}}]\circ G^{-1}\Bigr]_{i = 1}^{d}(x) \\
		& = \Bigl[ \bigl[(\sigma_{j_{2}}^{\top} (G_{i})'' + (G_{i})'(\sigma_{j_{2}})') \circ G^{-1}\bigr][(G')^{-1}\circ G^{-1}] [G' \circ G^{-1}][\sigma_{j_{1}} \circ G^{-1}] \Bigr]_{i = 1}^{d}(x) \\
		& = (((G'\sigma_{j_{2}})\circ G^{-1})' (G'\sigma_{j_{1}})\circ G^{-1})(x) = (((\sigma_{G})')_{j_{2}}(\sigma_{G})_{j_{1}})(x).
	\end{align*}
\end{proof}
\begin{Rem}
	It is very easy to see that if  $f \colon \R^{d} \to \R^{d\times d}$ which is $C^{1}$ on $\R^{d}\setminus M$ fulfils the commutativity condition outside $M$, then we also have
	\[
	\partial f_{j_{1}} f_{j_{2}} = \partial f_{j_{2}} f_{j_{1}}
	\]
	where $\partial f_{j}(x) = f_{j}'(x)$ if $x \in \R^{d}\setminus M$ and $\partial f_{j} = 0$, otherwise.
\end{Rem}
\subsection{Proof of Theorem~\ref{Thm1}}
Theorem~\ref{Thm2} is proven in much the same way as Theorem 2 in~\cite{MGY20}. For convenience of the reader we present the proof.
Due to Corolary \ref{ItoE} the stochastic process $G(X)$ is strong solution of an SDE with coefficents $\mu_{G}$ and $\sigma_{G}$, which fulfil the conditions (C) and (D) by Corollary \ref{gprop}. Using the Lipschitz continuity of $G^{-1}$ (see Proposition \ref{grep}(ii)) 
and Theorem \ref{Thm2} we obtain that there exist $c_{1},c_{2} \in (0,\infty)$ such that
\[
	\E(\norm{X - G^{-1}(\hat{Z}_{n})}_{\infty}^{p})^{\frac{1}{p}} \leq c_{1} \E(\norm{Z - \hat{Z}_{n}}_{\infty}^{p})^{\frac{1}{p}} \leq \frac{c_{2}}{n^{\frac{3}{4}-\delta}}
\]
Where $Z$ is the strong solution of the SDE
\begin{equation}\label{trafsde}
	\begin{aligned}
		dX_{t} &= \mu_{G}(X_{t})dt + \sigma_{G}(X_{t})dW_{t}, \quad t \in [0,1], \\
		X_{0} &= x_{0},
	\end{aligned}
\end{equation}
and $\hat{Z}_{n}$ is the quasi-Milstein-scheme associated with \eqref{trafsde} with step-size $\frac{1}{n}$.
The statement about the commutative setting follows immediately from Lemma \ref{com}.

\section{Examples} \label{Exmpl}
We now introduce a way to construct a variety of examples of SDEs which fulfil conditions (A) and (B). This is an adaptation of the Example from Section 4 in \cite{MGYR24} to fit our setting.
\begin{exs} \label{exgen}
	Let $\Theta \subseteq \R^{d}$ be a compact, orientable $C^{5}$-hypersurface. Let $n \in \N$ and let $K_{1}, \dots, K_{n}$ be open, disjoint sets such that
	$\R^{d}\setminus \Theta = \bigcup_{i = 1}^{n} K_{i}$. Let $f_{0},f_{1},\dots,f_{n} \colon \R^{d} \to \R^{d}$ be functions such that
	\begin{itemize}
		\item[i)] $(f_{0})_{|\Theta}$ is bounded,
		\item[ii)] there exists an open set $U \subseteq \R^{d}$ such that $\Theta \subseteq U$ and such that $(f_{i})_{|U}$ is $C^{4}$,
		\item[iii)] $(f_{i})_{|K_{i}}$ is $C^{1}$
		\item[iv)] $(f_{i})_{|K_{i}}$ and $(f_{i})'_{|K_{i}}$ are Lipschitz continuous
	\end{itemize}
	for all $i \in \{1,\dots,n \}$.
	Moreover let
	$\sigma \colon \R^{d} \to \R^{d \times d}$
	be a function such that
	\begin{itemize}
	  \item[v)] $\sigma$ is Lipschitz continuous 
	  \item[vi)] there exists a normal vector $\nor \colon \Theta \to \R^{d}$ along $\Theta$ with $\nor^{\top}(x)\sigma(x)  \neq 0$ for all $x \in \Theta$.
	  \item[vii)] $\sigma$ is $C^{1}$ and for all $j \in \{1,\dots,d \}$ the function $\sigma_{j}'$ is Lipschitz continuous
	  \item[viii)] there exists an $\epsilon \in (0,\infty)$ for which $\sigma_{|\Theta^{\epsilon}}$ is $C^{4}$
	\end{itemize}
	Then the coefficients
	\begin{equation} \label{Exm1}
	\mu \colon \R^{d} \to \R^{d}, \, x \mapsto f_{0}(x) \one_{\Theta}(x) + \sum_{i = 1}^{n} f_{i}(x) \one_{K_{i}}(x)
	\end{equation}
	and $\sigma$ fulfil conditions (A) and (B).
\end{exs}
\begin{proof}
	Since $\Theta$ is a compact $C^{5}$-hypersurface, we have $\reach(\Theta) > 0$ by Proposition 14 in \cite{T2008}. Let $\nor \colon \Theta \to \R^{d}$ be a normal vector along $\Theta$ with $\nor^{\top}(x)\sigma(x) \neq 0$ for all $x \in \Theta$. Since $\Theta$ is a $C^{5}$ hypersurface, by Lemma \ref{normalreg0} the function $\nor$ is $C^{4}$ which implies (A)(i) as $\Theta$ is compact. 
	
	Note that the function $(\nor\circ \pr_{\Theta})^{\top} \sigma$ is continuous on $\Theta^{\epsilon}$ for $\epsilon \in (0,\reach(\Theta))$ by Lemma \ref{projdist}(i) and iv). By vi) and the compactness of $\Theta$ we hence have $\inf_{x \in \Theta} \|\normal(x)^\top \sigma(x)\| > 0$. Thus (A)(ii) holds.  
	
	Let $\epsilon \in (0,\reach(\Theta))$ such that viii) is fulfilled for $\epsilon$. By Propositions \ref{normalreg0} and  \ref{projdist}(i) the function $\nor \circ \pr_{\Theta} \colon \Theta^{\epsilon} \to \R^{d}$ is $C^{4}$. We now prove, that (A)(iii) is fulfilled. Let $x \in \Theta$. By Lemma \ref{connect0} there exists an open set $U\subset \R^d$ such that $x\in U$ and $U\cap\Theta$ is connected.
	Define $\phi \colon \Theta \times \R \to \R^{d}$ by
	\[
	\phi(y,h) = y+ h \nor(y)
	\]  
	and put
	\[
	B_{1} = \phi((U\cap\Theta) \times (0,\varepsilon)), \quad 
	B_{2} = \phi((U\cap\Theta) \times (-\varepsilon,0)).
	\]
	Since $U$ is open there exists $\delta \in (0,\varepsilon)$ such that $B_{\delta}(x)  \subset U$.
	Clearly, 
	\[
	\{y + h \nor(y) \mid y \in B_{\delta}(x)\cap \Theta, \; h \in (0,\delta) \} \subset B_{1}
	\]
	and
	\[
	\{y + h \nor(y) \mid y \in B_{\delta}(x)\cap \Theta, \; h \in (-\delta,0) \} \subset B_{2}.
	\]
	Since $\phi$ is continuous and $(U\cap\Theta) \times (0,\varepsilon)$ as well as $(U\cap\Theta) \times (-\varepsilon,0)$ are connected, the sets $B_{1}$ and $B_{2}$ are also connected. Moreover, by  Lemma~\ref{fed0} we have $B_1, B_2 \subset \R^{d}\setminus\Theta=\bigcup_{i = 1}^{n} K_{i}$. Thus there 
	exist $i,j \in \{1,\dots,d \}$ such that $B_{1} \subseteq K_{i}$ and $B_{2} \subseteq K_{j}$.
	Thus for $y \in B_{\delta}(x) \cap \Theta$ by ii) we have 
	\begin{align*}
	\lim_{h \downarrow 0}\frac{\mu(y-h\nor(y))- \mu(y+h\nor(y)))}{2 \|\sigma(y)^\top\nor(y)\|^2} = \frac{f_{j}(y) - f_{i}(y)}{2 \|\sigma(y)^\top\nor(y)\|^2}.
	\end{align*}
	Thus the function
	\[
	\tilde{\alpha} \colon \Theta  \to \R^{d}, \, x \mapsto \lim_{h\downarrow 0}\frac{\mu(x-h\nor(x))- \mu(x+h\nor(x))}{2 \|\sigma(x)^\top\nor(x)\|^2}.
	\]
	is well defined.
	Since $f_{j}$, $f_{i}$, $\sigma$ and $\nor \circ \pr_{\Theta}$ are all $C^{3}$ on $B_{\delta}(x)$ the function
	\[
	\tilde{\alpha}_{|B_{\delta}(x)} =  \Bigl(\frac{f_{j}(\cdot) - f_{i}(\cdot)}{2 \|\sigma(\cdot)^\top\nor(\pr_{\theta}(\cdot))\|^2}\Bigr)_{|B_{\delta(x)}}
	\]
	is also $C^{4}$.
	Hence there exists an open set $U \subseteq \R^{d}$ with $\Theta \subseteq U$ and a $C^{4}$-function $\alpha \colon U \to \R^{d}$ with $\alpha_{| \Theta} = \tilde{\alpha}$.
	Since $\Theta$ is compact thus (A)(iii) is fulfilled. 
	
	Since $\mu_{K_{i}} = (f_{i})_{|K_{i}}$ is $C^{1}$ according to iii) and $\R^{d}\setminus \Theta = \cup_{i= 1}^{n} K_{i}$ with $K_{i}$ open we obtain that $\mu_{|\R^{d}\setminus \Theta}$ is $C^{1}$. With vii) this yields that (A)(iv) is fulfilled.
	
	Furthermore, according to Lemma \ref{top1} and the compactness of $\Theta$ there exists $\delta \in (0,\epsilon)$ such that $\Theta^{\delta} \subseteq U$. Moreover the set $\cl(\Theta^{\delta}) \subset \Theta^{\epsilon}$ is compact, thus $f_{1},\dots,f_{n}$ are bounded on $\Theta^{\delta}$ due to ii). It is now easy to see that (A)(v) is fulfilled.
	
	To prove (A)(vi) note that for $i, j \in \{1,\dots,n\} $ with $i \neq j$ and $x\in K_{i}$, $y \in K_{j}$ there exists no continuous function $\gamma \colon [0,1] \to \R^{d}\setminus \Theta$ with $\gamma(0) = x$ and $\gamma(1) = y$, as $\gamma([0,1]) \subseteq \R^{d} \setminus \Theta$ then could not be connected. Thus for the intrinsic metric $\rho$ on $\R^{d} \setminus \Theta$ and $x,y \in \R^{d}\setminus \Theta$ with $\rho(x,y) < \infty$ there exists a $j \in \{1,\dots,n\}$ such that $x,y \in K_{j}$. Thus the intrinsic Lipschitz continuity of $\mu_{|\R^{d}\setminus \Theta}$ and $\mu'_{|\R^{d}\setminus \Theta}$ follows from the Lipschitz continuity of $(f_{i})_{|K_{i}}$, $i \in \{1,\dots,n \}$ and $(f_{i})'_{|K_{i}}$, $i \in \{1,\dots,n \}$. With (vii), condition A(vi) follows. Clearly (B) is fulfilled.
\end{proof}

\section{Numerical results}\label{Num}
In this section we present numerical simulations for the performance
of the $L_p$-error  
\[
\varepsilon_{p,n}=\bigl(\EE\bigl[\|X_1-\eul_{n,1}\|^p\bigr]\bigr)^{1/p}
\]
of the Milstein scheme $\eul_n$ (see \eqref{MilSc}) with step-size $1/n$ at the final time point $1$.
We use $\eul_{N,1}$
with $N \in \N$ large as a reference estimate of $X_1$ and we approximate the $L_p$-error $\varepsilon_{p,n}$
by the 
corresponding
empirical $p$-th mean 
error
\begin{equation}\label{emperr}
	\widehat \varepsilon_{p,n}= \Bigl(\frac{1}{m} \sum_{i = 1}^{m} \|\eul_{N,1}^{i} - \eul_{n,1}^{i}\|^{p} \Bigr)^{1/p}
\end{equation}
based on $m$ 
Monte Carlo
repetitions
 $(\eul_{N,1}^{1}, \eul_{n,1}^{1}), \ldots, (\eul_{N,1}^{m}, \eul_{n,1}^{m})$ of 
 $(\eul_{N,1}, \eul_{n,1})$.

It has to be noted, that these numerical results do not directly illustrate our theoretical findings.
The implementation of the transformed Milstein-scheme suggested in the present article
requires direct computation of the inverse of the transform $G$, or numerical inversion of $G$, both of which turned
out to be infeasible for even the simple examples considered in this part.

However the numerical analysis of the Milstein-scheme for SDEs which fulfil assumptions (A) and (B) is useful to provide
evidence towards the following conjecture we have.

Assume that $\mu$ and $\sigma$ satisfy $(A)$ and $(B)$.
Then we conjecture that for every $p\in [1,\infty)$ and every $\delta\in (0,\infty)$ there exists $c\in(0, \infty)$ such that for all $n\in\N$, 
\begin{equation}
	\bigl(\EE\bigl[\|X-\eul_{n}\|_\infty^p\bigr]\bigr)^{1/p}\leq \frac{c}{n^{3/4-\delta}}. 
\end{equation} 

For our numerical experiments we consider a $2$-dimensional  SDE \eqref{SDE} with initial value $x_0\in\R^2$, drift coefficient  $\mu \colon \R^{2} \to \R^{2}$ given by
	\[
	\mu(x)= \begin{cases}
		\begin{pmatrix}
			-1 \\
			-1 \\ 
		\end{pmatrix}, & \text { if } \norm{x} < 2 \\[.6cm]
		\begin{pmatrix}
			1 \\
			1 \\
		\end{pmatrix}, & \text{ if } \norm{x} \geq 2
	\end{cases}
	\]
and diffusion coefficient
\[
\sigma \colon \R^{2} \to \R^{2 \times 2}, \; (x_{1},x_{2}) \mapsto  
\begin{pmatrix}
	x_{1} & x_{1} \\
	x_{2} & x_{2} 
\end{pmatrix}.
\]
Thus, the drift coefficient $\mu$ is discontinuous 
on the set
$\Theta = \{x \in \R^{2} \colon \norm{x} = 2 \}$ and the diffusion coefficient $\sigma$ fulfils the commutativity condition.

We show that $\mu$ and $\sigma$ satisfy the conditions (A) and (B). 
Note
that $\mu$ is of the form \eqref{Exm1} with 
$n=2$, the  sets $K_1$ and $K_2$ given by
\[
K_{1} = \{x \in \R^{2} \colon \norm{x} < 2 \},\qquad
K_{2} = \{x \in \R^{2} \colon \norm{x} > 2 \}
\]
and  the functions $f_0, f_1, f_2\colon\R^2\to\R^2$ given by
\begin{align*}
	f_{0}(x) =f_2(x)= \begin{pmatrix}
		1 \\
		1
	\end{pmatrix}, \qquad 
	f_{1} (x)= \begin{pmatrix}
		-1 \\
		-1
	\end{pmatrix}.
\end{align*}
Clearly,  $\Theta$ is a compact $C^{\infty}$-hypersurface,  $K_{1}$ and $K_{2}$ are open and disjoint sets covering $\R^{2}$ and the functions $f_0,f_1$ and $f_2$ fulfil the conditions (i)--(iv) 
from
Example \ref{exgen}. Moreover, clearly $\sigma$ is  $C^{\infty}$ and Lipschitz continuous with Lipschitz continuous derivative.
Thus the conditions (v),(vii) and (viii) from Example \ref{exgen} hold. Since $\sigma$ has compact support, condition (iv) from Example \ref{exgen} holds as well. 
Finally, 
$\nor \colon \Theta \to \R^{2}, \, x \mapsto \frac{1}{2}x$, is a normal vector along $\Theta$ and for all $x \in \Theta$ we have
\[
\nor(x)^\top\sigma(x) = 	\begin{pmatrix}
	x^\top x & x^\top x \\
\end{pmatrix}  \neq 0,
\]
which proves the condition (vi) from Example \ref{exgen}. \\
\hspace*{0.5cm}
We now choose $n = 2^7,2^8,\dots,2^{15}$, $N = 2^{17}$  and $m = 10^6$. Furthermore we select $x_{0} = (0,2)^\top$ as our initial value, i.e. the SDE \eqref{SDE} starts at time $0$ at a point of discontinuity  of $\mu$. Figure \ref{fig2} shows the plot of
a realization of
 the  empirical $L_2$-error $\widehat \varepsilon_{2,n}$  of the method versus the number of time-steps $n$ for the Euler-Maruyama scheme on the left hand side and for the Milstein-scheme on the right hand side. The empirical $L_2$-error rate for Euler  is  $0.53$ and for Milstein it is $0.82$.\\
\begin{figure}[H]
	\centering
	\caption{Empirical $L_2$-error vs. number of time steps: $a = -1, b = 1, x_{0} = (0,2)^\top$}
	\makebox[\textwidth]{\includegraphics[scale=0.23]{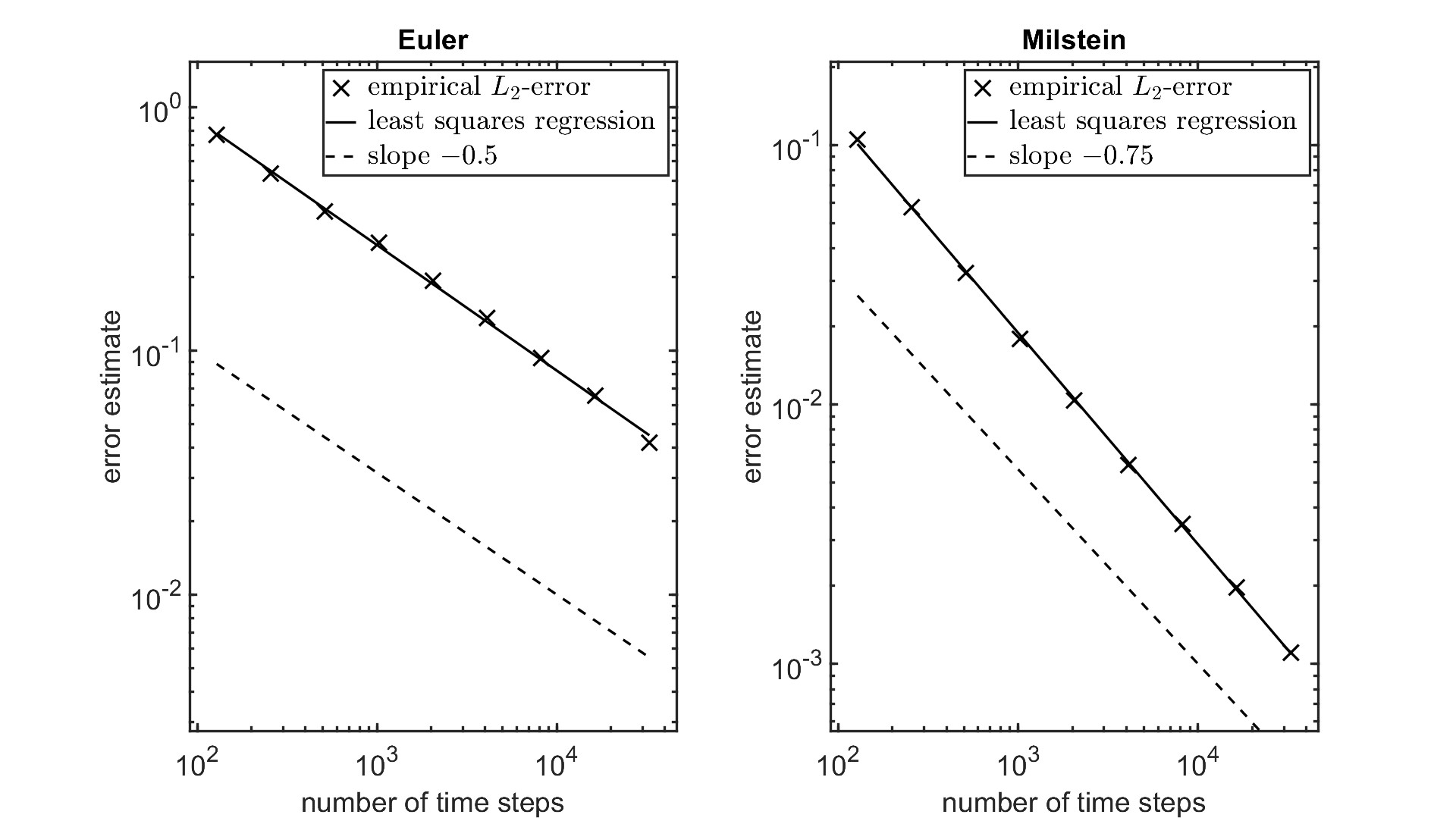}}
	\label{fig2}
\end{figure}
For the SDE currently under consideration,  the empirical $L_p$-error rate decreases 
with increasing $p$, see Table \ref{tab2}. A similar behaviour was observed for the Euler-Maruyama scheme under
assumptions weaker than (A) and (B) in \cite{MGYR24}. We refer to this article for possible explanations of this behaviour. 
\begin{table}[H] 
\label{tab2}
\caption {Empirical $L_p$-error rates for different $p$}
\begin{center}
\begin{tabular}{|c|c|c|c|c|}
\hline
	$p$ & 1 & 2 & 4 & 8 \\
	\hline
	Euler & 0.53 & 0.51 & 0.49 & 0.48 \\ \hline
	Milstein & 0.84 & 0.82 & 0.79 &  0.79   \\
	\hline
\end{tabular}
\end{center}
\end{table}
It can be seen that the observed $L_{p}$-error rates lie above $3/4$, giving some hint towards our conjecture being true.
\subsection*{Acknowledgements:}

The author would like to thank Thomas Müller-Gronbach and Larisa Yaroslavtseva
for useful critiques, significantly improving this article.

\section{Appendix - Basic facts on hypersurfaces}\label{Appendix}

The article \cite{MGYR24} includes an appendix which covers a variety  of known basic facts on hypersurfaces, distance functions, normal vectors, projections and intrinsic Lipschitz continuity which are also used in this text, mainly in section \ref{Proofs}.
We provide a list of the results used in this text here for the convenience of the reader and refer to \cite{MGYR24} and the references therein for proofs of these results.
\begin{Lem} \cite[Lemma 21]{MGYR24} \label{connect0} 

	Let  $d\in\N$, let 
	$\emptyset\not=M \subset \R^{d}$ be a $C^{0}$-hypersurface, let $x \in M$ and let $A\subset \R^d$ be open with $x\in A$. Then there exists an open set $\widetilde A\subset A$ with $x\in \widetilde A$ such that $\widetilde A\cap M$ is connected. 
\end{Lem}
\begin{Lem} \cite[Lemma 25]{MGYR24} \label{fed0}
	Let $d\in\N$, let  
	$\emptyset\neq M \subset \R^{d}$ 
	 be a $C^{1}$-hypersurface of positive reach and let
	$x \in M$. If $u \in T_{x}(M)^{\perp}$ 
	and
	$\norm{u} < \reach(M)$ then  $x+u \in \text{Unp}(M)$ and $\text{pr}_{M}(x+u) = x$.
\end{Lem}
\begin{Lem} \cite[Lemma 28]{MGYR24} \label{projdist} 
	Let  $d\in\N$, let 
	$k \in \N \cup \{\infty \}$, let $\emptyset\not=M \subset\R^{d}$ be a $C^{k}$-hypersurface of positive reach and let $\eps\in (0,\reach(M))$. Then the following statements hold true.
	\begin{itemize}
		\item[(i)] $\pr_M $ is a $C^{k-1}$-mapping on $M^\eps$.
			\item[(ii)] 	 $d(\cdot,M)$ is a $C^{k}$-function on $M^\eps\setminus M$.
		\item[(iii)] If $k\ge 2$ then for all $x\in M^\eps$ and all $v\in\R^d$,
		\[
		(\pr_{M}'(x))v\in T_{ \text{pr}_{M}(x)}(M).
		\]
		In particular,
		\[
		(x-\pr_M(x))^\top \pr'_M(x) = 0.
		\]
	\end{itemize} 
\end{Lem}
\begin{Lem} \cite[Lemma 29]{MGYR24} \label{sides0}
	Let  $d\in\N$, let  
	$\emptyset\neq M \subset \R^{d}$ be a $C^{1}$-hypersurface of positive reach, let $\nor\colon M \to \R^{d}$ be a normal vector along $M$.
Then for all $\eps \in (0,\reach(M))$, the sets 
	\[Q_{\varepsilon, +} = \{x +\lambda \nor(x) \mid x \in M, \, \lambda \in (0,\eps)\}
	\] 
	and
	\[
	Q_{\varepsilon, -} = \{x+\lambda \nor(x) \mid x \in M,\, \lambda \in (-\eps,0) \}
	\] 
	are open and disjoint and we have $M^{\eps}\setminus M = Q_{\varepsilon, +} \cup Q_{\varepsilon, -}$.
\end{Lem}
\begin{Lem} \cite[Lemma 30]{MGYR24} \label{normalreg0}
	Let $d\in\N$, let 
$k \in \N \cup \{\infty \}$, let $\emptyset\not=M \subset\R^{d}$ be a $C^{k}$-hypersurface and let $\nor\colon M\to \R^d$ be a normal vector along $M$. Then $\nor$ is a $C^{k-1}$-function. Moreover, if $k\ge 2$ then for all $x \in M$ and all $v \in T_{x}(M)$ we have $\nor'(x)v \in T_{x}(M)$.
\end{Lem}
\begin{Lem} \cite[Lemma 33]{MGYR24} \label{Lipconnec0}
	Let  $d\in\N$, let
$\emptyset\neq A \subset \R^{d}$ and let $\rho_A$ be the intrinsic metric for $A$. Then   $\|x-y\| \leq \rho_A(x,y)$ for all $x,y \in A$ and $\|x-y\| = \rho_A(x,y)$ for all $x,y \in A$ with $\overline{x,y}\subset A$. In particular, if $A$ is convex then $\rho_A$ coincides with the Euclidean distance.
\end{Lem}
\begin{Lem} \cite[Lemma 34]{MGYR24} \label{lipimpl0} 
	Let $d, m \in \N$, 
	let $\emptyset \neq A \subset \R^{d}$ and let $f\colon A\to \R^{m}$ be a function.
	\begin{itemize}
		\item [(i)] If $f$ is Lipschitz continuous then $f$ is intrinsic Lipschitz continuous.
		\item[(ii)] If $A$ is convex then $f$ is intrinsic Lipschitz continuous with intrinsic Lipschitz constant $L$ if and only if $f$ is Lipschitz continuous with Lipschitz constant $L$.
		\item[(iii)] If $A$ is open and $f$ is  intrinsic Lipschitz continuous then $f$ is locally Lipschitz continuous and, in particular, $f$ is continuous.
	\end{itemize}  
\end{Lem}
\begin{Lem} \cite[Lemma 36]{MGYR24} \label{top1}
	Let  $d\in \N$, let 
	$\emptyset \not=A\subset\R^d$
	be open and let $K\subset A$ be compact. Then there exists $\eps\in (0,\infty)$ such that $K^\eps \subset A$.
\end{Lem}
\begin{Lem} \cite[Lemma 39]{MGYR24} \label{diffintr0} 
	Let $d,m\in \N$, let
	$\emptyset\neq A \subset \R^{d}$ be open and let $f\colon  A \to \R^{m}$ be differentiable with $\|f'\|_{\infty} < \infty$. Then $f$ is intrinsic Lipschitz continuous with intrinsic Lipschitz constant $\|f'\|_{\infty}$.
\end{Lem}
\begin{Lem} \cite[Lemma 40]{MGYR24} \label{comp0}
	Let $d,m \in \N$, let $\emptyset\neq B \subset \R^{d}$ and $\emptyset\neq A \subset \R^{m}$ be open, let $g\colon  B \to A$
	be intrinsic Lipschitz continuous with intrinsic Lipschitz constant $L_{g}$, let $f \colon  A \to \R^{m}$ be intrinsic Lipschitz continuous with intrinsic Lipschitz constant $L_{f}$.
	Then $f \circ g\colon  B \to \R^{m}$ is intrinsic Lipschitz continuous with intrinsic Lipschitz constant $L_{f} L_{g}$.
\end{Lem}
\begin{Lem} \cite[Lemma 41]{MGYR24} \label{productnew0}
	Let $d,m,k,\ell\in \N$, let $\emptyset\neq A,B,
	C
	 \subset \R^{d}$ with $A,B\subset C$, let $f\colon C \to \R^{m \times k}$ and $g\colon C \to \R^{k\times\ell}$ be intrinsic Lipschitz continuous on $A$ and bounded on $B$, and let $f$ be constant on $C\setminus B$. Then $f  g\colon C\to\R^{m\times \ell}$ is intrinsic Lipschitz continuous on $A$.
\end{Lem}
\begin{Lem} \cite[Lemma 42]{MGYR24} \label{a0} 
	Let $d\in\N$, let $\emptyset\neq M \subset \R^{d}$ be a $C^{1}$-hypersurface 
	of positive reach
	and let $\nor\colon M \to \R^{d}$ be a normal vector along $M$. Let $f\colon \R^{d} \to \R^{d}$ be piecewise Lipschitz continuous with exceptional set $M$ and
	assume that for all $x \in M$, the limit $\lim_{h \to 0} f(x+h\nor(x))$  exists and coincides with $f(x)$. Then $f$ is continuous.
\end{Lem}

\bibliographystyle{acm}
\bibliography{bibfile}

\end{document}